\documentclass[12pt]{amsart}
\usepackage[textwidth=16cm, textheight=22cm]{geometry}
\usepackage{hyperref}
\newcommand{\arxiv}[1]{\href{http://www.arXiv.org/abs/#1}{arXiv:#1}}
\usepackage{graphicx}
\usepackage{latexsym}

\newtheorem{theorem}{Theorem}[section]
\newtheorem{lemma}[theorem]{Lemma}

\newtheorem{proposition}[theorem]{Proposition}
\newtheorem{maintheorem}{Main Theorem}

\newtheorem{corollary}[theorem]{Corollary}
\theoremstyle{remark}

\def\R{\mathbb{R}}

\def\boole{\mathbb{B}}

\def\Rp{\R_+}
\def\Rpn{\Rp^n}
\def\Rpnn{\Rp^{n\times n}}
\def\Tc{T_c}

\def\critindices{N_c}

\def\supp{\operatorname{supp}}
\def\spann{\operatorname{span}}
\def\core{\operatorname{core}}
\def\diag{\operatorname{diag}}
\def\maxmu{\Tilde{\mu}}

\def\goodclass{{\mathcal S}}
\def\crit{{\mathcal C}}
\def\digr{{\mathcal G}}
\def\itoj{i\;-\;j}
\def\jtoi{j\;-\;i}
\def\mutonu{\mu\;-\;\nu}

\def\coremax{\core_{\oplus}}
\def\corenon{\core_+}
\def\vmax{V_{\oplus}}
\def\vnon{V_+}
\def\spanmax{\spann_{\oplus}}
\def\spannon{\spann_+}
\def\Lmax{\Lambda_{\oplus}}
\def\Lnon{\Lambda_+}
\def\rhomax{\rho^{\oplus}}
\def\rhonon{\rho^+}
\def\mrho{M_{\rho}}
\def\amrho{A_{\rho}}
\def\amrhomrho{A_{\mrho\mrho}}

\begin{document}

\title{Two cores of a nonnegative matrix}
\thanks{This research was supported by EPSRC grant RRAH15735.
Serge\u{\i} Sergeev also acknowledges the support of RFBR-CNRS grant
11-0193106 and RFBR grant 12-01-00886. Bit-Shun Tam acknowledges the
support of National Science Council of the Republic of China
(Project No. NSC 101-2115-M-032-007)}

\author{Peter Butkovi{\v{c}}}
\address{Peter Butkovi{\v{c}}, University of Birmingham,
School of Mathematics, Watson Building, Edgbaston B15 2TT, UK}
\email{P.Butkovic@bham.ac.uk}

\author{Hans Schneider}
\address{Hans Schneider, University of Wisconsin-Madison,
Department of Mathematics, 480 Lincoln Drive, Madison WI 53706-1313,
USA} \email{hans@math.wisc.edu}

\author{Serge\u{\i} Sergeev}
\address{Serge\u{\i} Sergeev, University of Birmingham,
School of Mathematics, Watson Building, Edgbaston B15 2TT, UK}
\email{sergeevs@maths.bham.ac.uk}

\author{Bit-Shun Tam}
\address{Bit-Shun Tam, Tamkang University, Department of Mathematics, Tamsui, Taiwan 25137,
R.O.C. }
\email{bsm01@mail.tku.edu.tw}

\begin{abstract}
We prove that the sequence of eigencones (i.e., cones of nonnegative
eigenvectors) of positive powers $A^k$ of a nonnegative square
matrix $A$  is periodic both in max algebra and in nonnegative
linear algebra. Using an argument of Pullman, we also show that the
Minkowski sum of the eigencones of powers of $A$ is equal to the
core of $A$ defined as the intersection of nonnegative column spans
of matrix powers, also in max algebra. Based on this, we describe
the set of extremal rays of the  core.

The spectral theory of matrix powers and the theory of matrix core
is developed in max algebra and in nonnegative linear algebra
simultaneously wherever possible, in order to unify and compare both
versions of the same theory.


\end{abstract}

\maketitle

\section{Introduction}

The nonnegative reals $\Rp$ under the usual multiplication give rise to two
semirings with addition defined in two ways: first with the usual addition, and
second where the role of addition is played by maximum.
Thus we consider
the properties of nonnegative matrices with entries in two semirings, the semiring of nonnegative numbers with usual addition
and multiplication called ``{\bf nonnegative algebra}'',
and the semiring called ``{\bf max(-times) algebra}''.

Our chief object of study is the {\bf core} of a nonnegative matrix
$A$.  This concept was introduced  by Pullman in ~\cite{Pul-71}, and
is defined as the intersection of the cones generated by the columns
of matrix powers $A^k$. Pullman provided a geometric approach to the
Perron-Frobenius theory of nonnegative matrices based on the
properties of the core.  He  investigated the action of a matrix on
its core showing that it is bijective and that the extremal rays of
the core can be partitioned into periodic orbits. In other words,
extremal rays of the core of $A$ are nonnegative eigenvectors of the
powers of $A$ (associated with positive eigenvalues).

One of the main purposes of the present paper is to extend Pullman's
core to max algebra, thereby investigating the periodic sequence of
eigencones of max-algebraic matrix powers.  However, following the
line of~\cite{BSS,BSS-zeq,KSS}, we develop the theory in max algebra
and nonnegative algebra simultaneously, in order to emphasize common
features as well as differences, to provide general (simultaneous)
proofs where this is possible. We do not aim to obtain new results,
relative to~\cite{Pul-71,TS-94}, on the usual core of a nonnegative
matrix. However, our unifying approach leads in
some cases (e.g., Theorem~\ref{t:tam-schneider} (iii)) to new and
more elementary proofs than those given previously. Our motivation
is closely related to the Litvinov-Maslov correspondence
principle~\cite{LM-98}, viewing the idempotent mathematics (in
particular, max algebra) as a ``shadow'' of the ``traditional''
mathematics over real and complex fields.

\if{
Pullman's core can be also seen as closely related to the
limits of powers of nonnegative matrices. However it is a different
concept. Consider the simple example
$$ \begin{pmatrix}
1 & 0 \\
0 & 0.5
\end{pmatrix}.
$$
Then, for any nonnegative $x$, $A^kx$ will tend to a multiple of $(
1,\  0)^T$ while the core of $A$ is the entire nonnegative orthant
${\mathbb R}^2_+$. }\fi

To the authors' knowledge, the core of a nonnegative matrix has not
received much attention in linear algebra. However, a more detailed
study has been carried out by Tam and Schneider~\cite{TS-94}, who
extended the concept of core to linear mappings preserving a proper
cone. The case when the core is a polyhedral (i.e., finitely
generated) cone was examined in detail in~\cite[Section 3]{TS-94},
and the results were applied to study the case of nonnegative matrix
in~\cite[Section 4]{TS-94}. This work has found further applications
in the theory of dynamic systems acting on the path space of a
stationary Bratteli diagram. In particular,
Bezuglyi~et~al.~\cite{BKMS} describe and exploit a natural
correspondence between ergodic measures and extremals of the core of
the incidence matrix of such a diagram.

On the other hand, there is much more literature on the related but
distinct question of the  limiting sets of  homogeneous and
non-homogeneous Markov chains in nonnegative algebra; see the books
by Hartfiel~\cite{Har:02} and Seneta~\cite{Sen:81} and, e.g., the
works of Chi~\cite{Chi-96} and Sierksma~\cite{Sie-99}. In max
algebra, see the results on the ultimate column span of matrix
powers for irreducible matrices (\cite[Theorem 8.3.11]{But:10},
\cite{Ser-09}), and by Merlet~\cite{Mer-10} on the invariant max
cone of non-homogeneous matrix products.

The theory of the core relies on the behaviour of matrix powers. In
the nonnegative algebra, recall the works of
Friedland-Schneider~\cite{FS-80} and Rothblum-Whittle~\cite{RW-82}
(on the role of distinguished classes which we call ``spectral
classes'', algebraic and geometric growth rates, and various
applications). The theory of max-algebraic matrix powers is similar.
However, the max-algebraic powers have a well-defined periodic
ultimate behaviour starting after sufficiently large time. This
ultimate behaviour has been known since the work of
Cuninghame-Green~\cite[Theorem 27-9]{CG:79},
Cohen~et~al.~\cite{CDQV-83} (irreducible case), and is described in
greater generality and detail, e.g., by Akian, Gaubert and
Walsh~\cite{AGW-05}, Gavalec~\cite{Gav:04}, De Schutter~\cite{BdS},
and the authors~\cite{But:10,Ser-11<attr>, SS-11} of the present
paper. In particular, the Cyclicity Theorem of
Cohen~et~al.~\cite{BCOQ,But:10,CDQV-83,HOW:05}) implies that
extremals of the core split into periodic orbits for any irreducible
matrix (see Subsection 4.2 below)\footnote{
In fact, many of the cited works and monographs
like~\cite{BCOQ,But:10,Gav:04,HOW:05} are written in the setting of
{\bf max-plus algebra}. However, this algebra is isomorphic to the
max algebra considered here, so the results can be readily
translated to the present (max-times) setting.}.

Some results on the eigenvectors of max-algebraic matrix powers have
been obtained by Butkovi\v{c} and
Cuninghame-Green~\cite{But:10,CGB-07}. The present paper also aims
to extend and complete the research initiated in that work.

This paper is organized as follows.
In Section~\ref{s:prel} we introduce the basics of irreducible and reducible Perron-Frobenius theory in max
algebra and in nonnegative linear algebra.
In Section~\ref{s:key} we formulate the two key results of this
paper. The first key result is Main Theorem~\ref{t:core} stating
that the matrix core equals to the Minkowski sum of the eigencones
of matrix powers (that is, for each positive integer $k$, we take
the sum of the eigencones associated with $A^k$, and then we sum
over all $k$). The second key result is Main
Theorem~\ref{t:periodicity} stating that the sequence of eigencones
of matrix powers is periodic and defining the period. This section
also contains a table of notations used throughout the paper.
Section~\ref{s:core} is devoted to the proof of Main
Theorem~\ref{t:core}, taking in ``credit'' the result of Main
Theorem~\ref{t:periodicity} (whose proof is deferred to the end of
the paper).
In Section~\ref{s:sameaccess} we explain the relation between
spectral classes of different matrix powers, and how the eigencones
associated with general eigenvalues can be reduced to the case of
the greatest eigenvalue, see in particular
Theorems~\ref{t:samespectrum} and~\ref{t:reduction}.
In Section~\ref{s:extremals} we describe extremals of the core in
both algebras extending~\cite[Theorem~4.7]{TS-94}, see
Theorem~\ref{t:tam-schneider}. Prior to this result we formulate the
Frobenius-Victory Theorems~\ref{t:FVnonneg} and~\ref{t:FVmaxalg}
giving a parallel description of extremals of eigencones in both
algebras. In Section~\ref{s:eigencones}, our first goal is to show
that the sequence of eigencones of matrix powers in max algebra is
periodic, comparing this result with the case of nonnegative matrix
algebra, see Theorem~\ref{t:girls}.
Then we study the inclusion relation on eigencones and deduce Main
Theorem~\ref{t:periodicity}.
The key results are illustrated by a
pair of examples in Section~\ref{s:examples}.

\section{Preliminaries}
\label{s:prel}

\subsection{Nonnegative matrices and associated graphs}
\label{ss:nonneg}

In this paper we are concerned only with nonnegative eigenvalues and
nonnegative eigenvectors of a nonnegative matrix. In order to bring
our terminology into line with the corresponding theory for max
algebra we use the terms eigenvalue and  eigenvector in a
restrictive fashion appropriate to our semiring point of view. Thus
we shall call $\rho$ an {\em eigenvalue} of a nonnegative matrix $A$
(only) if there is a nonnegative eigenvector $x$ of $A$ for $\rho$.
Further $x$ will be called an {\em eigenvector} (only) if it is
nonnegative. (In the literature $\rho$ is called a distinguished
eigenvalue and $x$ a distinguished eigenvector of $A$.) For
$x\in\Rpn$, the {\em support} of $x$, denoted by $\supp(x)$, is the
set of indices where $x_i>0$.

In this paper we are led to state the familiar Perron-Frobenius theorem in slightly
unusual terms: An irreducible nonnegative
matrix $A$ has a unique eigenvalue denoted by $\rhonon(A)$, which is positive (unless $A$ is the $1\times 1$ matrix $0$). Further, the eigenvector $x$ associated with  $\rhonon(A)$ is essentially unique, that is all eigenvectors are multiples of $x$.
The nonnegative multiples of $x$ constitute the cone of eigenvectors (in the above sense)
$\vnon(A,\rhonon(A))$ associated with $\rhonon(A)$.

A general (reducible) matrix $A\in\Rpnn$ may have several nonnegative eigenvalues with associated cones of nonnegative eigenvectors
({\em eigencones}), and $\rhonon(A)$ will denote the biggest such eigenvalue, in general. Eigenvalue
$\rhonon(A)$ is also called the {\em principal eigenvalue}, and $\vnon(A,\rhonon(A))$ is called the
{\em principal eigencone}.

Recall that a subset $V\subseteq\Rpn$ is called a (convex) cone if 1) $\alpha v\in V$ for
all $v\in V$ and $\alpha\in\Rp$, 2) $u+v\in V$ for $u,v\in V$. Note that cones in the nonnegative orthant can be
considered as ``subspaces'', with respect to the semiring of nonnegative numbers
(with usual addition and multiplication). In this vein,
a cone $V$ is said to be {\em generated} by $S\subseteq\Rpn$ if each $v\in V$ can be represented
as a nonnegative combination $v=\sum_{x\in S} \alpha_x x$ where only finitely many
$\alpha_x\in\Rp$ are different from
zero. When $V$ is generated (we also say ``spanned'') by $S$,
this is denoted $V=\spannon(S)$.   A vector $z$ in a cone $V$ is called an {\em extremal}, if $z=u+v$ and $u,v\in V$ imply $z=\alpha_u u=\alpha_v v$
for some scalars $\alpha_u$ and $\alpha_v$. Any closed cone in $\Rpn$ is generated by its extremals; in particular,
this holds for any finitely generated cone.

Let us recall some basic notions related to (ir)reducibility, which we use also
in max algebra. With a matrix $A=(a_{ij})\in\Rp^{n\times n}$
we associate a weighted (di)graph $\digr(A)$ with the set of nodes
$N=\{1,\dots,n\}$ and set of edges~$E\subseteq N\times N$ containing a pair~$(i,j)$ if and only
if~$a_{ij}\neq 0$; the weight of an edge~$(i,j)\in E$ is defined to be~$w(i,j):=a_{ij}$.
A graph with just one node and no
edges will be called {\em trivial}. A graph with at least one node and at least one edge
will be called {\em nontrivial}.

A path $P$ in $\digr(A)$ consisting\footnote{In our terminology, a
path can visit some nodes more than once.} of the edges
$(i_0,i_1),(i_1,i_2),\ldots,(i_{t-1},i_t)$ has {\em length}
$l(P):=t$ and {\em weight} $w(P):=w(i_0,i_1) \cdot w(i_1,i_2) \cdots
w(i_{t-1},i_t)$, and is called an $\itoj$ path if $i_0=i$ and
$i_t=j$. $P$ is called a {\em cycle} if $i_0=i_t$. $P$ is an {\em
elementary cycle}, if, further, $i_k\neq i_l$ for all
$k,l\in\{1,\ldots,t-1\}$.

Recall that $A=\left( a_{ij}\right) \in \Rp^{n\times n}$ is
irreducible if $\digr(A)$ is trivial or for any $i,j\in \{1,\ldots, n\}$ there is an
$\itoj$ path. Otherwise $A$ is reducible.

Notation $A^{\times k}$ will stand for the usual $k$th power of a nonnegative matrix.

\subsection{Max algebra}
\label{ss:max}

By max algebra we understand the set of nonnegative numbers $\Rp$
where the role of addition is played by taking maximum of two
numbers: $a\oplus b:=\max(a,b)$, and the multiplication is as in the
usual arithmetics. This is carried over to matrices and vectors like
in the usual linear algebra so that for two matrices $A=(a_{ij})$
and $B=(b_{ij})$ of appropriate sizes, $(A\oplus
B)_{ij}=a_{ij}\oplus b_{ij}$ and $(A\otimes B)_{ij}=\bigoplus_k
a_{ik}b_{kj}$.
Notation $A^{\otimes k}$ will stand for the $k$th max-algebraic power.

In max algebra, we have the following analogue of a convex cone.  A set
$V\subseteq\Rpn$ is called a {\em max cone} if 1) $\alpha v\in V$
for all $v\in V$ and $\alpha\in\Rp$, 2) $u\oplus v\in V$ for $u,v\in
V$. Max cones are also known as idempotent
semimodules~\cite{KM:97,LM-98}. A max cone $V$ is said to be {\em
generated} by $S\subseteq\Rpn$ if each $v\in V$ can be represented
as a max combination $v=\bigoplus_{x\in S} \alpha_x x$ where only
finitely many (nonnegative) $\alpha_x$ are different from zero. When
$V$ is generated (we also say ``spanned'') by $S$, this is denoted
$V=\spanmax(S)$. When $V$ is generated by the columns of a matrix
$A$, this is denoted $V=\spanmax(A)$. This cone is closed with
respect to the usual Euclidean topology~\cite{BSS}.

A vector $z$ in a max cone $V\subseteq\Rpn$ is called an {\em extremal} if $z=u\oplus v$ and $u,v\in V$ imply
$z=u$ or $z=v$. Any finitely generated max cone is generated by its extremals, see Wagneur~\cite{Wag-91} and~\cite{BSS,GK-07} for recent extensions. 


The {\em maximum cycle geometric mean} of~$A$ is defined by
\begin{equation}\label{mcgm}
\lambda(A)=\max\{ w(C)^{1/l(C)}\colon C \text{ is a cycle in }\digr(A)\} \enspace.
\end{equation}
The cycles with the cycle geometric mean equal to $\lambda(A)$ are
called {\em critical}, and the nodes and the edges of $\digr(A)$
that belong to critical cycles are called {\em critical}. The set of
critical nodes is denoted by $N_c(A)$, the set of critical edges by
$E_c(A)$, and these nodes and edges give rise to the {\em critical
graph} of $A$, denoted by $\crit(A) = (N_c(A), E_c(A))$.
A maximal strongly connected subgraph of 
$\crit(A)$ is called a strongly connected component of $\crit(A)$.
Observe that $\crit(A)$, in general, consists of several nontrivial
strongly connected components, and that it never has any edges
connecting different strongly connected components.

\if{
The {\em critical graph} of $A$, denoted by $\crit(A)$, consists of all nodes and edges belonging to the cycles
which attain the maximum in~\eqref{mcgm}. The set of such nodes will be called {\em critical} and denoted $N_c$;
the set of such edges will be called {\em critical} and denoted $E_c$.
Observe that the critical graph
consists of several strongly connected subgraphs of $\digr(A)$. Maximal such subgraphs are the
{\em strongly connected components} of $\crit(A)$, and there are no
critical edges connecting different
strongly connected components of $\crit(A)$.
}\fi

If for $A\in\Rpnn$ we have
$A\otimes x=\rho x$
with $\rho\in\Rp$ and a nonzero $x\in\Rpn$, then $\rho$ is a {\em max(-algebraic) eigenvalue} and $x$ is a {\em max(-algebraic)
eigenvector} associated with $\rho$. The set of
max eigenvectors $x$ associated with~$\rho$, with the zero vector adjoined to it, is a max cone. It is denoted
by $\vmax(A,\rho)$.

An irreducible $A\in\Rpnn$ has a unique max-algebraic eigenvalue
equal to $\lambda(A)$~\cite{BCOQ,But:10,CG:79,HOW:05}. In general
$A$ may have several max eigenvalues, and the greatest of them
equals $\lambda(A)$.  The greatest max eigenvalue will also be
denoted by $\rhomax(A)$ (thus $\rhomax(A)=\lambda(A)$), and called
the {\em principal max eigenvalue} of $A$. In the irreducible case,
the unique max eigenvalue $\rhomax(A)=\lambda(A)$ is also called the
{\em max(-algebraic) Perron root}. When max algebra and nonnegative
algebra are considered simultaneously (e.g., Section~\ref{s:key}),
the principal eigenvalue is denoted by $\rho(A)$.

Unlike in nonnegative algebra, there is an explicit description of
$\vmax(A,\rhomax(A))$, see Theorem~\ref{t:FVmaxalg}. This
description uses the {\em Kleene star}
\begin{equation}
\label{def:kleenestar}
A^*=I\oplus A\oplus A^{\otimes 2}\oplus A^{\otimes 3}\oplus\ldots.
\end{equation}
Series~\eqref{def:kleenestar} converges if and only if
$\rhomax(A)\leq 1$, in which case $A^*=I\oplus A\oplus\ldots\oplus
A^{\otimes(n-1)}$~\cite{BCOQ,But:10,HOW:05}. Note that if
$\rhomax(A)\neq 0$, then $\rhomax(A/\rhomax(A))=1$, hence
$(A/\rhomax(A))^*$ always converges.

The {\em path interpretation} of max-algebraic matrix powers
$A^{\otimes l}$ is that each entry $a^{\otimes l}_{ij}$ is equal to
the greatest weight of $\itoj$ path with length $l$. Consequently,
for $i\neq j$, the entry $a^*_{ij}$ of $A^*$ is equal to the
greatest weight of an $\itoj$ path (with no length restrictions).

\subsection{Cyclicity and periodicity}
\label{ss:cycl} Consider a nontrivial strongly connected graph
$\digr$ (that is, a strongly connected graph with at least one node
and one edge). Define its {\em cyclicity} $\sigma$ as the gcd of the
lengths of all elementary cycles. It is known that for any vertices
$i,j$ there exists a number $l$ such that $l(P)\equiv l(\text{mod}\
\sigma)$
for all $\itoj$ paths $P$.

When the length of an $\itoj$ path is a multiple of $\sigma$ (and
hence we have the same for all $\jtoi$ paths), $i$ and $j$ are said
to belong to the same {\em cyclic class}. When the length of this
path is $1$ modulo $\sigma$ (in other words, if $l(P)-1$ is a
multiple of $\sigma$), the cyclic class of $i$ (resp., of $j$) is
{\em previous} (resp., {\em next}) with respect to the class of $j$
(resp., of $i$).
See~\cite[Chapter 8]{But:10} and~\cite{BR,Ser-09,Ser-11<attr>} for
more information. Cyclic classes are also known as {\em components
of imprimitivity}~\cite{BR}.

The cyclicity of a trivial graph is defined to be $1$, and the
unique node of a trivial graph is defined to be its only cyclic
class.

We define the cyclicity of a (general) graph containing several
strongly connected components to be the lcm of the cyclicities
of the components. 

For a graph $\digr=(N,E)$ with $N=\{1,\ldots,n\}$, define the {\em associated matrix} $A=(a_{ij})\in\{0,1\}^{n\times n}$
by
$
a_{ij}=1\Leftrightarrow (i,j)\in E.
$
This is a matrix over the Boolean semiring $\boole:=\{0,1\}$, where addition
is the disjunction and multiplication is the conjunction operation. This semiring is a subsemiring of
max algebra, so that it is possible to consider the associated matrix as a matrix in max algebra whose entries are either $0$ or $1$.


For a graph $\digr$ and any $k\geq 1$, define $\digr^k$ as a graph that has the
same vertex set as $\digr$ and $(i,j)$ is an edge of $\digr^k$ if and only if there
is a path of length $k$ on $\digr$ connecting $i$ to $j$. Thus, if a Boolean matrix
$A$ is associated with $\digr$, then the Boolean matrix power $A^{\otimes k}$ is associated
with $\digr^k$. Powers of Boolean matrices (over the Boolean semiring) are a topic of independent interest, see~Brualdi-Ryser~\cite{BR}, Kim~\cite{Kim}. We will need the
following observation.

\begin{theorem}[cf. {\cite[Theorem 3.4.5]{BR}}]
\label{t:brualdi} Let $\digr$ be a strongly connected graph with
cyclicity $\sigma$.
\begin{itemize}
\item[{\rm (i)}] $\digr^k$ consists of gcd $(k,\sigma)$
nontrivial strongly connected components not accessing each
other. If $\digr$ is nontrivial, then so are all the components
of $\digr^k$.
\item[{\rm (ii)}] The node set of each  component of $\digr^k$ consists of $\sigma/$(gcd$(k,\sigma))$
cyclic classes of $\digr$. 
\end{itemize}
\end{theorem}

\begin{corollary}
\label{c:div-boolean} Let $\digr$ be a strongly connected graph with
cyclicity $\sigma$, and let $k,l\geq 1$. Then gcd$(k,\sigma)$
divides gcd$(l,\sigma)$ if and only if $\digr^k$ and $\digr^l$ are
such that the node set of every component of $\digr^l$ is contained
in the node set of a component of $\digr^k$.
\end{corollary}
\begin{proof} Assume that $\digr$ is nontrivial.\\
``If''. Since the node set of each component of $\digr^k$ consists of
$\sigma/$gcd$(k,\sigma)$ cyclic classes of $\digr$ and is the disjoint union of the
node sets of certain components of $\digr^l$, and the node set of each component of
$\digr^l$ consists of $\sigma/$gcd$(l,\sigma)$ cyclic classes of $\digr$, it follows
that the node set of each component of $\digr^k$ consists of
$\frac{\sigma}{\text{gcd}(k,\sigma)}/\frac{\sigma}{\text{gcd}(l,\sigma)}=
\frac{\text{gcd}(l,\sigma)}{\text{gcd}(k,\sigma)}$ components of $\digr^l$. Therefore,
gcd$(k,\sigma)$ divides gcd$(l,\sigma)$.

``Only if.'' Observe that the node sets of the compopnents $\digr^k$
and $\digr^{\text{gcd}(k,\sigma)}$ (or $\digr^l$ and
$\digr^{\text{gcd}(l,\sigma)}$) are the same: since gcd$(k,\sigma)$
divides $k$, each component of $\digr^{\text{gcd}(k,\sigma)}$ splits
into several components of $\digr^k$, but the total number of
components is the same (as
gcd$($gcd$(k,\sigma),\sigma)=$gcd$(k,\sigma)$), hence their node
sets are the same. The claim follows since the node set of each
component of $\digr^{\text{gcd}(k,\sigma)}$ splits into several
components of $\digr^{\text{gcd}(l,\sigma)}$.
\end{proof}

Let us formally introduce the definitions related to periodicity and
ultimate periodicity of sequences (whose elements are of arbitrary
nature). A sequence $\{\Omega_k\}_{k\geq 1}$ is called {\em
periodic} if there exists an integer $p$ such that $\Omega_{k+p}$ is
identical with $\Omega_k$ for all $k$. The least such $p$ is called
the {\em period} of $\{\Omega_k\}_{k\geq 1}$. A sequence
$\{\Omega_k\}_{k\geq 1}$ is called {\em ultimately periodic} if the
sequence $\{\Omega_k\}_{k\geq T}$ is periodic for some $T\geq 1$.
The least such $T$ is called the {\em periodicity threshold} of
$\{\Omega_k\}_{k\geq 1}$.

The following observation is crucial in the theory of
Boolean matrix powers.

\begin{theorem}[Boolean Cyclicity~\cite{Kim}]
\label{t:BoolCycl} Let $\digr$ be a strongly connected graph on $n$
nodes, with cyclicity $\sigma$.
\begin{itemize}
\item [{\rm (i)}] The sequence
$\{\digr^k\}_{k\geq 1}$ is ultimately periodic with the period
$\sigma$. The periodicity threshold, denoted by $T(\digr)$, does
not exceed $(n-1)^2+1$.
\item[{\rm (ii)}] If $\digr$ is nontrivial, then for $k\geq T(\digr)$ and a
multiple of $\sigma$, $\digr^k$ consists of $\sigma$ complete subgraphs
not accessing each other.
\end{itemize}
\end{theorem}
For brevity, we will refer to $T(\digr)$ as the periodicity
threshold of $\digr$.
We have the following two max-algebraic
extensions of Theorem~\ref{t:BoolCycl}.

\begin{theorem}[Cyclicity Theorem, Cohen~et~al.~\cite{CDQV-83}]
\label{t:Cycl1} Let $A\in\Rpnn$ be irreducible, let $\sigma$ be the
cyclicity of $\crit(A)$ and $\rho:=\rhomax(A)$. Then the sequence
$\{(A/\rho)^{\otimes k}\}_{k\geq 1}$ is ultimately periodic with
period $\sigma$.
\end{theorem}

\begin{theorem}[Cyclicity of Critical Part, Nachtigall~\cite{Nacht}]
\label{t:nacht} Let $A\in\Rpnn$, $\sigma$ be the cyclicity of
$\crit(A)$ and $\rho:=\rhomax(A)$. Then the sequences
$\{(A/\rho)^{\otimes k}_{i\cdot}\}_{k\geq 1}$ and
$\{(A/\rho)^{\otimes k}_{\cdot i}\}_{k\geq 1}$, for $i\in N_c(A)$,
are ultimately periodic with period $\sigma$. The greatest of their
periodicity thresholds, denoted by $\Tc(A)$, does not exceed $n^2$.
\end{theorem}

\if{
The least number $T$ (resp. $\Tc$) satisfying the condition of
Theorem~\ref{t:Cycl1} (resp. Theorem~\ref{t:nacht}) is denoted by
$T(A)$ (resp. $\Tc(A)$), and called the {\em cyclicity threshold}
(resp. the {\em critical cyclicity threshold}) of $A$.
}\fi

Theorem~\ref{t:Cycl1} is standard~\cite{BCOQ,But:10,HOW:05}, and
Theorem~\ref{t:nacht} can also be found as~\cite[Theorem
8.3.6]{But:10}. Here $A_{i\cdot}$ (resp. $A_{\cdot i}$) denote the
$i$th row (resp. the $i$th column) of $A$.

When the sequence  $\{(A/\rho)^{\otimes k}\}_{k\geq 1}$ (resp. the
sequences $\{(A/\rho)^{\otimes k}_{i\cdot}\}_{k\geq 1}$,\\
$\{(A/\rho)^{\otimes k}_{\cdot i}\}_{k\geq 1}$) are ultimately
periodic, we also say that the sequence $\{A^{\otimes k}\}_{k\geq
1}$ (resp. $\{A^{\otimes k}_{i\cdot}\}_{k\geq 1}$, $\{A^{\otimes
k}_{\cdot i}\}_{k\geq 1}$) is {\em ultimately periodic with growth
rate $\rho$}.

Let us conclude with a well-known number-theoretic result concerning the
coin problem of Frobenius, which we see as basic for both Boolean and max-algebraic cyclicity.

\begin{lemma}[e.g.,{\cite[Lemma 3.4.2]{BR}}]
\label{l:schur} Let $n_1,\ldots, n_m$ be integers such that\\
gcd$(n_1,\ldots,n_m)=k$. Then there exists a number $T$ such that
for all integers $l$ with $kl\geq T$, we have $kl=t_1n_1+\ldots
+t_mn_m$ for some $t_1,\ldots,t_m\geq 0$.
\end{lemma}

\subsection{Diagonal similarity and visualization}
\label{ss:viz}

For any $x\in\Rpn$, we can define $X=\diag(x)$ as the {\em diagonal
matrix} whose diagonal entries are equal to the corresponding
entries of $x$, and whose off-diagonal entries are zero. If $x$ does
not have zero components, the diagonal similarity scaling $A\mapsto
X^{-1}AX$ does not change the weights of cycles and eigenvalues
(both nonnegative and max); if $z$ is an eigenvector of $X^{-1}AX$
then $Xz$ is an eigenvector of $A$ with the same eigenvalue. This
scaling does not change the critical graph
$\crit(A)=(N_c(A),E_c(A))$. Observe that $(X^{-1}AX)^{\otimes
k}=X^{-1}A^{\otimes k}X$, also showing that the periodicity
thresholds of max-algebraic matrix powers (Theorems~\ref{t:Cycl1}
and~\ref{t:nacht})
do not change after scaling. Of course, we also have
$(X^{-1}AX)^{\times k}=X^{-1}A^{\times k}X$ in nonnegative algebra.
The technique of nonnegative scaling can be traced back to the works
of Fiedler-Pt\'ak~\cite{FP-67}.

When working with the max-algebraic matrix powers, it is often
convenient to ``visualize'' the powers of the critical graph. Let
$A$ have $\lambda(A)=1$. A diagonal similarity scalling $A\mapsto
X^{-1}AX$ is called a {\em strict visualization
scaling}~\cite{But:10,SSB} if the matrix $B=X^{-1}AX$ has
$b_{ij}\leq 1$, and moreover, $b_{ij}=1$ if and only if $(i,j)\in
E_c(A)(=E_c(B))$. Any matrix $B$ satisfying these properties is
called {\em strictly visualized}.

\begin{theorem}[Strict Visualization~\cite{But:10,SSB}]
\label{t:strictvis} For each $A\in\Rpnn$ with $\rhomax(A)=1$ {\rm
(}that is, $\lambda(A)=1${\rm )}, there exists a strict
visualization scaling.
\end{theorem}

If $A=(a_{ij})$ has all entries $a_{ij}\leq 1$, then we define the
Boolean matrix $A^{[1]}$ with entries
\begin{equation}
\label{def:A1}
a_{ij}^{[1]}=
\begin{cases}
1, &\text{if $a_{ij}=1$},\\
0, &\text{if $a_{ij}<1$}.
\end{cases}
\end{equation}
If $A$ has all entries $a_{ij}\leq 1$ then
\begin{equation}
\label{e:unimatrix}
(A^{\otimes k})^{[1]}=(A^{[1]})^{\otimes k}.
\end{equation}
Similarly if a vector $x\in\Rpn$ has $x_i\leq 1$, we define $x^{[1]}$ having
$x^{[1]}_i=1$ if $x_i=1$ and $x^{[1]}_i=0$ otherwise. Obviously if $A$ and $x$ have all entries not exceeding
$1$ then $(A\otimes x)^{[1]}=A^{[1]}\otimes x^{[1]}$. 

If $A$ is strictly visualized, then $a_{ij}^{[1]}=1$ if and only if
$(i,j)$ is a critical edge of $\digr(A)$. Thus $A^{[1]}$ can be
treated as the associated matrix of $\crit(A)$ (disregarding the
formal difference in dimension). We now show that $\crit(A^{\otimes
k})=\crit(A)^k$ and that any power of a strictly visualized matrix
is strictly visualized.


\begin{lemma}[cf.~\cite{CGB-07}, {\cite[Prop.~3.3]{Ser-09}}]
\label{l:CAk}
Let $A\in\Rpnn$ and $k\geq 1$.
\begin{itemize}
\item[{\rm (i)}]  $\crit(A)^k=\crit(A^{\otimes k})$.
\item[{\rm (ii)}] If $A$ is strictly visualized, then so is $A^{\otimes k}$.
\end{itemize}
\end{lemma}
\begin{proof}
Using Theorem~\ref{t:strictvis}, we can assume without loss of
generality that $A$ is strictly visualized. Also note that both in
$\crit(A^{\otimes k})$ and in $\crit(A)^k$, each node has ingoing
and outgoing edges, hence for part (i) it suffices to prove that the
two graphs have the same set of edges.

Applying Theorem~\ref{t:brualdi} (i) to every component of
$\crit(A)$, we obtain that $\crit(A)^k$ also consists of several
isolated nontrivial strongly connected graphs. In particular, each
edge of $\crit(A)^k$ lies on a cycle, so $\crit(A)^k$ contains
cycles. Observe that $\digr(A^{\otimes k})$ does not have edges with
weight greater than $1$, while all edges of $\crit(A)^k$ have weight
$1$, hence all cycles of $\crit(A)^k$ have weight $1$. As
$\crit(A)^k$ is a subgraph of $\digr(A^{\otimes k})$, this shows
that $\rhomax(A^{\otimes k})=\lambda(A^{\otimes k})=1$ and that all
cycles of $\crit(A)^k$ are critical cycles of $\digr(A^{\otimes
k})$. Since each edge of $\crit(A)^k$ lies on a critical cycle, all
edges of $\crit(A)^k$ are critical edges of $\digr(A^{\otimes k})$.

$\digr(A^{\otimes k})$ does not have edges with weight greater than
$1$, hence every edge of $\crit(A^{\otimes k})$ has weight $1$.
Equation~\eqref{e:unimatrix} implies that if $a_{ij}^{\otimes k}=1$
then there is a path from $i$ to $j$ composed of the edges with
weight $1$. Since $A$ is strictly visualized, such edges are
critical. This shows that if $a_{ij}^{\otimes k}=1$ and in
particular if $(i,j)$ is an edge of $\crit(A^{\otimes k})$, then
$(i,j)$ is an edge of $\crit(A)^k$.  Hence $A^{\otimes k}$ is
strictly visualized, and all edges of $\crit(A^{\otimes k})$ are
edges of $\crit(A)^k$.

Thus $\crit(A^{\otimes k})$ and $\crit(A)^k$ have the same set of
edges, so $\crit(A^{\otimes k})=\crit(A)^k$ (and we also showed that
$A^{\otimes k}$ is strictly visualized).
\end{proof}

Let $T(\crit(A))$ be the greatest periodicity threshold of the
strongly connected components of $\crit(A)$. The following corollary
of Lemma~\ref{l:CAk} will be required in Section~\ref{s:eigencones}.

\begin{corollary}
\label{TcaTcrit} Let $A\in\Rpnn$. Then $\Tc(A)\geq T(\crit(A))$.
\end{corollary}
\if{
\begin{proof}
Let $\sigma$ be the cyclicity of $\crit(A)$. Assume that $A$ is strictly visualized.
We are going to show that if $(A^k)_{\cdot i}=A^{k+\sigma}_{\cdot i}$ for all $i\in N_c(A)$
then $(\crit(A))^k=(\crit(A))^{k+\sigma}$. Indeed, Lemma~\ref{l:CAk} implies that $(\crit(A))^k=(\crit(A))^{k+\sigma}$ is
equivalent to $(A^[1])^{\otimes k}= (A^[1])^{\otimes (k+\sigma)}$ and that all nonzero entries of
these matrices are in the critical columns and rows. As $(A^k)_{\cdot i}=A^{k+\sigma}_{\cdot i}$ for all $i\in N_c(A)$
is sufficient for $(A^[1])^{\otimes k}= (A^[1])^{\otimes (k+\sigma)}$, the claim follows.
\end{proof}
}\fi

\subsection{Frobenius normal form}

Every matrix $A=(a_{ij})\in \Rp^{n\times n}$ can be transformed by
simultaneous permutations of the rows and columns in almost linear
time to
a {\em Frobenius normal form}~\cite{BP,BR} 
\begin{equation}
\left(
\begin{array}{cccc}
A_{11} & 0 & ... & 0 \\
A_{21} & A_{22} & ... & 0 \\
... & ... & A_{\mu\mu} & ... \\
A_{r1} & A_{r2} & ... & A_{rr}%
\end{array}%
\right) ,  \label{fnf}
\end{equation}%
where $A_{11},...,A_{rr}$ are irreducible square submatrices of $A$.
They correspond to the sets of nodes $N_1,\ldots,N_r$ of the
strongly connected components of $\digr(A)$. Note that
in~\eqref{fnf} an edge from a node of $N_{\mu}$ to a node of
$N_{\nu}$ in $\digr(A)$ may exist only if $\mu\geq \nu.$

Generally, $A_{KL}$ denotes the submatrix of $A$ extracted from the
rows with indices in $K\subseteq \{1,\ldots,n\}$ and columns with
indices in $L\subseteq \{1,\ldots,n\}$, and $A_{\mu\nu}$ is a
shorthand for $A_{N_{\mu}N_{\nu}}$. Accordingly, the subvector
$x_{N_{\mu}}$ of $x$ with indices in $N_{\mu}$ will be written as
$x_{\mu}$.

If $A$ is in the Frobenius Normal Form \eqref{fnf} then the {\em
reduced graph}, denoted by $R(A)$, is the (di)graph whose nodes
correspond to $N_{\mu}$, for $\mu=1,\ldots,r$, and the set of arcs
is $ \{(\mu,\nu);(\exists k\in N_{\mu})(\exists \ell \in
N_{\nu})a_{k\ell }>0\}. $ In max algebra and in nonnegative algebra,
the nodes of $R(A)$ are {\em marked} by the corresponding
eigenvalues (Perron roots), denoted by
$\rhomax_{\mu}:=\rhomax(A_{\mu\mu})$ (max algebra),
$\rhonon_{\mu}:=\rhonon(A_{\mu\mu})$ (nonnegative algebra), and by
$\rho_{\mu}$ when both algebras are considered simultaneously.

By a {\em class} of $A$ we mean a node $\mu$ of the reduced graph
$R(A)$. It will be convenient to attribute to class $\mu$ the node
set and the edge set of $\digr(A_{\mu\mu})$, as well as the
cyclicity and other parameters, that is, we will say ``nodes of
$\mu$'', ``edges of $\mu$'', ``cyclicity of $\mu$'',
etc.\footnote{The sets $N_{\mu}$ are also called classes, in the
literature. To avoid the confusion, we do not follow this in the
present paper.}

A class $\mu$ is trivial if $A_{\mu\mu}$ is the $1\times 1$ zero matrix.
Class $\mu$ {\em accesses} class
$\nu$, denoted $\mu\to\nu$, if $\mu=\nu$ or if there exists a $\mutonu$ path in $R(A)$.
A class is called {\em initial}, resp. {\em final}, if it is not accessed by, resp. if it does not
access, any other class. Node $i$ of $\digr(A)$ accesses class $\nu$, denoted by $i\to\nu$, if $i$ belongs to a class $\mu$ such that $\mu\to\nu$.

Note that simultaneous permutations of the rows and columns of $A$
are equivalent to calculating $P^{-1}AP,$ where $P$ is a permutation
matrix.  Such transformations do not change the eigenvalues, and the
eigenvectors before and after such a transformation may only differ
by the order of their components. Hence we will assume without loss
of generality that $A$ is in Frobenius normal form~(\ref{fnf}). Note
that a permutation bringing matrix to this form is (relatively) easy
to find~\cite{BR}. We will refer to the transformation $A\mapsto
P^{-1}AP$ as {\em permutational similarity}.

\subsection{Elements of the Perron-Frobenius theory}
\label{ss:pfelts}

In this section we recall the spectral theory of reducible matrices in
max algebra and in nonnegative linear algebra.
All results are standard: the
nonnegative part goes back to Frobenius~\cite{Fro-1912}, Sect.~11,
and the max-algebraic counterpart is due to Gaubert~\cite{Gau:92},
Ch.~IV (also see~\cite{But:10} for other references).

A class $\nu$ of $A$ is called a {\em spectral class} of $A$
associated with eigenvalue $\rho\neq 0$, or sometimes
$(A,\rho)$-spectral class for short, if
\begin{equation}
\label{e:speclass}
\begin{split}
&\rhomax_{\nu}=\rho,\ \text{and}\ \mu\to\nu\ \text{implies}\ \rhomax_{\mu}\leq\rho^{\oplus}_{\nu}\ \text{(max algebra)},\\
&\rhonon_{\nu}=\rho,\ \text{and}\ \mu\to\nu,\mu\neq\nu\ \text{implies}\ \rho^{+}_{\mu}<\rho^{+}_{\nu}\ \text{(nonnegative algebra)}.
\end{split}
\end{equation}
In both algebras, note that there may be several spectral classes
associated with the same eigenvalue.

In nonnegative algebra, spectral classes are called distinguished
classes~\cite{Sch-86}, and there are also semi-distinguished classes
associated with distinguished generalized eigenvectors of order two
or more~\cite{TS-00}. However, these vectors are not contained in
the core\footnote{For a polyhedral cone, the core of the
cone-preserving map does not contain generalized eigenvectors of
order two or more~{\cite[Corollary 4.3]{TS-94}}.}. Also, no suitable
max-algebraic analogue of generalized eigenvectors is known to us.


If all classes of $A$ consist of just one element, then the
nonnegative and max-algebraic Perron roots are the same. In this
case, the spectral classes in nonnegative algebra are also spectral
in max algebra. However, in general this is not so. In particular,
for a nonnegative matrix $A$, a cycle of $\digr(A)$ attaining the
maximum cycle geometric mean $\rhomax(A)=\lambda(A)$ need not lie in
a strongly connected component corresponding to a class with
spectral radius $\rhonon(A)$. This is because, if $A_1$, $A_2$ are
irreducible nonnegative matrices such that
$\rhonon(A_1)<\rhonon(A_2)$, then we need not have
$\rhomax(A_1)<\rhomax(A_2)$. For example, let $A$ be the $3\times 3$
matrix of all $1$'s, and let
$B(\epsilon)=(3/2,\epsilon,\epsilon)^T(3/2,\epsilon,\epsilon)$. Then
$\rhonon(A)=3$, $\rhonon(B(\epsilon))=9/4+2\epsilon^2$, so
$\rhonon(A)>\rhonon(B(\epsilon))$ for sufficiently small
$\epsilon>0$, but $\rhomax(B(\epsilon))=9/4>1=\rhomax(A)$.

Denote by $\Lnon(A)$, resp. $\Lmax(A)$, the set of {\bf nonzero}
eigenvalues of $A\in\Rpnn$ in nonnegative linear algebra, resp. in
max algebra. It will be denoted by $\Lambda(A)$ when both algebras
are considered simultaneously, as in the following standard
description.

\begin{theorem}[{\cite[Th.~4.5.4]{But:10}}, {\cite[Th.~3.7]{Sch-86}}]
\label{t:spectrum} Let $A\in\Rpnn$. Then
$\Lambda(A)=\{\rho_{\nu}\neq 0\colon  \nu\ \text{is spectral}\}$.
\end{theorem}

Theorem~\ref{t:spectrum} encodes the following {\bf two} statements:
\begin{equation}
\label{e:spectrum2}
\Lmax(A)=\{\rhomax_{\nu}\neq 0\colon  \nu\ \text{is spectral}\},\quad
\Lnon(A)=\{\rhonon_{\nu}\neq 0\colon  \nu\ \text{is spectral}\},
\end{equation}
where the notion of spectral class is defined in two different ways
by~\eqref{e:speclass}, in two algebras.

In both algebras, for each $\rho\in\Lambda(A)$ define 
\begin{equation}
\label{amrho}
\begin{split}
&\amrho:=\rho^{-1}
\begin{pmatrix}
0 & 0\\
0 & \amrhomrho
\end{pmatrix},\ \text{where}\\
& \mrho:=\{i\colon i\to\nu,\; \nu\ \text{is $(A,\rho)$-spectral}\}\,\enspace .
\end{split}
\end{equation}

By ``$\nu$ is $(A,\rho)$-spectral'' we mean that $\nu$ is a spectral
class of $A$ with $\rho_{\nu}=\rho$. The next proposition, holding
both in max algebra and in nonnegative algebra, allows us to reduce
the case of arbitrary eigencone to the case of principal eigencone.
Here we assume that $A$ is in Frobenius normal form.


\begin{proposition}[\cite{But:10,Gau:92,Sch-86}]
\label{p:vamrho}
For $A\in\Rpnn$ and each $\rho\in\Lambda(A)$, we have
$V(A,\rho)=V(\amrho,1)$, where $1$ is the principal eigenvalue
of $\amrho$.
\end{proposition}

For a parallel description of extremals of eigencones\footnote{In
nonnegative algebra, \cite[Th. 3.7]{Sch-86} immediately describes
both spectral classes and eigencones associated with any eigenvalue.
However, we prefer to split the formulation, following the
exposition of~\cite{But:10}. An alternative simultaneous exposition
of spectral theory in both algebras can be found in~\cite{KSS}.} in
both algebras see Section~\ref{ss:FV}.

In max algebra, using Proposition~\ref{p:vamrho}, we define the {\em
critical graph associated with} $\rho\in\Lmax(A)$ as the critical
graph of $A_{\rho}$. By a {\em critical component of $A$} we mean a
strongly connected component of the critical graph associated with
some $\rho\in\Lmax(A)$.
In max algebra, the role of spectral classes of $A$ is rather played by
these critical components, which will be (in analogy
with classes of Frobenius normal form) denoted by $\Tilde{\mu}$, with the node set $N_{\Tilde{\mu}}$.
See Section~\ref{ss:critcomp}.


\section{Notation table and key results}
\label{s:key}

The following notation table shows how various objects
are denoted in nonnegative algebra, max algebra and when both algebras are considered simultaneously.

\begin{equation*}
\begin{array}{cccc}
& \text{Nonnegative} & \text{Max} & \text{Both}\\
\text{Sum} & \sum & \bigoplus & \sum\\
\text{Matrix power} & A^{\times t} & A^{\otimes t} & A^t\\
\text{Column span}   & \spannon(A) & \spanmax(A) & \spann(A)\\
\text{Perron root} & \rhonon_{\mu} & \rhomax_{\nu} & \rho_{\mu}\\
\text{Spectrum (excl. $0$)} & \Lnon(A) & \Lmax(A) & \Lambda(A)\\
\text{Eigencone} & \vnon(A,\rhonon) & \vmax(A,\rhomax) & V(A,\rho)\\
\text{Sum of eigencones} & \vnon^{\Sigma}(A) & \vmax^{\Sigma}(A) & V^{\Sigma}(A)\\
\text{Core} & \corenon(A) & \coremax(A) & \core(A)
\end{array}
\end{equation*}
In the case of max algebra, we also have
the critical graph $\crit(A)$ (with related concepts and notation),
not used in nonnegative algebra.

The core and the sum of eigencones appearing in the table have not
been formally introduced. These are the two central notions of this
paper, and we now introduce them for both algebras simultaneously.

The {\em core} of a nonnegative matrix $A$ is defined as the
intersection of the column spans (in other words, images) of its
powers:
\begin{equation}
\label{def:core}
\core(A):=\cap_{i=1}^{\infty} \spann(A^i).
\end{equation}

The ({\em Minkowski}) {\em sum of eigencones} of a nonnegative
matrix $A$ is the cone consisting of all sums of vectors in all
$V(A,\rho)$:
\begin{equation}
\label{va-def}
V^{\Sigma}(A):=\sum_{\rho\in\Lambda(A)} V(A,\rho)
\end{equation}
If $\Lambda(A)=\emptyset$, which happens when $\rho(A)=0$, then we
assume that the sum on the right-hand side is $\{0\}$.

The following notations can be seen as the ``global'' definition of
cyclicity in nonnegative algebra and in max algebra.
\begin{itemize}
\item[1.] Let $\sigma_{\rho}$ be the
the lcm of all cyclicities of
spectral classes associated with $\rho\in\Lnon(A)$ ({\bf nonnegative algebra}),
or the cyclicity of critical graph associated with $\rho\in\Lmax(A)$ ({\bf max algebra}).
\item[2.] Let $\sigma_{\Lambda}$ be the lcm of all $\sigma_{\rho}$ where $\rho\in\Lambda(A)$.
\end{itemize}
The difference between the definitions of $\sigma_{\rho}$ in max
algebra and in nonnegative algebra comes from the corresponding
versions of Perron-Frobenius theory. In particular, let $A\in\Rpnn$
be an irreducible matrix. While in nonnegative algebra the eigencone
associated with the Perron root of $A$ is always reduced to a single
ray, the number of (appropriately normalized) extremals of the
eigencone of $A$ in max algebra is equal to the number of critical
components, so that there may be up to $n$ such extremals.


One of the key results of the present paper relates the core with
the sum of eigencones. The nonnegative part of this result can be
found in Tam-Schneider~\cite[Th.~4.2, part~(iii)]{TS-94}.
\begin{maintheorem}
\label{t:core}
Let $A\in\Rpnn$. Then
$$\core(A)=\sum_{k\geq 1,\rho\in\Lambda(A)} V(A^k,\rho^k) =V^{\Sigma}(A^{\sigma_{\Lambda}}).$$
\end{maintheorem}

The main part of the proof is given in Section~\ref{s:core}, for
both algebras simultaneously. However, this proof takes in
``credit'' some facts, which we will have to show. First of all, we
need the equality
\begin{equation}
\label{e:samespectrum}
\Lambda(A^k)=\{\rho^k\colon \rho\in\Lambda(A)\}.
\end{equation}
This simple relation between $\Lambda(A^k)$ and $\Lambda(A)$, which
can be seen as a special case of~\cite[Theorem 3.6(ii)]{KSS}, will
be also proved below as Corollary~\ref{c:samespectrum}.

To complete the proof of Main Theorem~\ref{t:core} we also have to
study the periodic sequence of eigencones of matrix powers and their
sums. On this way we obtain the following key result, both in max
and nonnegative algebra.

\begin{maintheorem}
\label{t:periodicity} Let $A\in\Rpnn$. Then
\begin{itemize}
\item[{\rm (i)}] $\sigma_{\rho}$, for $\rho\in\Lambda(A)$, is the period of the
sequence $\{V(A^k,\rho^k)\}_{k\geq 1}$, and
$V(A^k,\rho^k)\subseteq
V(A^{\sigma_{\rho}},\rho^{\sigma_{\rho}})$ for all $k\geq 1$;
\item[{\rm (ii)}] $\sigma_{\Lambda}$ is the period of the sequence
$\{V^{\Sigma}(A^k)\}_{k\geq 1}$, and $V^{\Sigma}(A^k)\subseteq V^{\Sigma}(A^{\sigma_{\Lambda}})$ for all $k\geq 1$.
\end{itemize}
\end{maintheorem}

Main Theorem~\ref{t:periodicity} is proved in
Section~\ref{s:eigencones} as a corollary of
Theorems~\ref{t:girls-major} and~\ref{t:girls-minor}, where the
inclusion relations between eigencones of matrix powers are studied
in detail.


Theorem~\ref{t:tam-schneider}, which
gives a detailed description
of extremals of both cores, can be also seen as a key result of this paper.
However, it is too long to be formulated in advance.

\section{Two cores}
\label{s:core}

\subsection{Basics}
\label{ss:basics} In this section we investigate the core of a
nonnegative matrix defined by~\eqref{def:core}. In the main
argument, we consider the cases of max algebra and nonnegative
algebra simultaneously.

One of the most elementary and useful properties of the intersection
in~\eqref{def:core} is that, actually,
\begin{equation}
\label{e:monotonic}
\spann(A)\supseteq\spann(A^2)\supseteq\spann(A^3)\supseteq\ldots
\end{equation}

Generalizing an argument of Pullman~\cite{Pul-71} we will prove that
\begin{equation}
\label{e:maingoal}
\core(A)=\sum_{k\geq 1} V^{\Sigma}(A^k)=\sum_{k\geq 1,\rho\in\Lambda(A)} V(A^k,\rho^k)
\end{equation}
also in max algebra.

Note that the following inclusion is almost immediate.

\begin{lemma}
\label{l:natural}
$\sum_{k\geq 1} V^{\Sigma}(A^k)\subseteq\core (A)$.
\end{lemma}
\begin{proof}
$x\in V(A^k,\rho)$ implies that $A^k x=\rho x$ and hence $x=\rho^{-t} A^{kt}x$ for all
$t\geq 1$ (using the invertibility of multiplication). Hence
$x\in\bigcap_{t\geq 1} \spann A^{kt}=\bigcap_{t\geq 1} \spann(A^t)$.
\end{proof}

So it remains to show the opposite inclusion
\begin{equation}
\label{e:maingoal2}
\core(A)\subseteq\sum_{k\geq 1} V^{\Sigma}(A^k).
\end{equation}

Let us first treat the trivial case $\rho(A)=0$, i.e.,
$\Lambda(A)=\emptyset$. There are only trivial classes in the
Frobenius normal form, and $\digr(A)$ is acyclic. This implies
$A^k=0$ for some $k\geq 1$. In this case $\core(A)=\{0\}$, the sum
on the right-hand side is $\{0\}$ by convention,
so~\eqref{e:maingoal} is the trivial "draw" $\{0\}=\{0\}$.

\subsection{Max algebra: cases of ultimate periodicity}
\label{ss:maxalg-easy}

In max algebra, unlike the nonnegative algebra, there are wide classes of matrices
 for which~\eqref{e:maingoal2} and~\eqref{e:maingoal} follow almost immediately. We list some of them below.\\
$\goodclass_1:$ {\em Irreducible matrices}.\\
$\goodclass_2:$ {\em Ultimately periodic matrices.} This is when the
sequence $\{A^{\otimes k}\}_{k\geq 1}$ is ultimately periodic with a
growth rate $\rho$ (in other words, when the sequence
$\{(A/\rho)^{\otimes k}\}_{k\geq 1}$ is ultimately periodic). As
shown by Moln\'arov\'a-Pribi\v{s}~\cite{MP-00}, this happens if and
only if the Perron roots of all nontrivial classes
of $A$ equal $\rhomax(A)=\rho$. \\
$\goodclass_3:$ {\em Robust matrices.} For any vector $x\in\Rpn$ the
orbit $\{A^{\otimes k}\otimes x\}_{k\geq 1}$ hits an eigenvector of
$A$, meaning that $A^{\otimes T}\otimes x$ is an eigenvector of $A$
for some $T$. This implies that the whole remaining part
$\{A^{\otimes k}\otimes x\}_{k\geq T}$ of the orbit (the ``tail'' of
the orbit) consists of multiples of that eigenvector $A^{\otimes
T}\otimes x$. The notion
of robustness was introduced and studied in~\cite{BCG-09}.\\
$\goodclass_4:$ {\em Orbit periodic matrices:} For any vector
$x\in\Rpn$ the orbit $\{A^{\otimes k}\otimes x\}_{k\geq 1}$ hits an
eigenvector of $A^{\otimes \sigma_x}$ for some $\sigma_x$, implying
that the remaining ``tail'' of the orbit $\{(A^{\otimes k}\otimes
x\}_{k\geq 1}$ is periodic with some growth rate. See~\cite[Section
7]{SS-11} for
characterization.\\
$\goodclass_5:$ {\em Column periodic matrices.} This is when for all $i$ we have
$(A^{\otimes (t+\sigma_i)})_{\cdot i}=\rho_i^{\sigma_i} A^{\otimes t}_{\cdot i}$ for all large enough $t$ and some
$\rho_i$ and $\sigma_i$.

Observe that $\goodclass_1\subseteq\goodclass_2\subseteq\goodclass_4\subseteq\goodclass_5$ and
$\goodclass_3\subseteq\goodclass_4$.
Indeed, $\goodclass_1\subseteq\goodclass_2$ is the Cyclicity Theorem~\ref{t:Cycl1}.
For the inclusion $\goodclass_2\subseteq\goodclass_4$ observe that, if
$A$ is ultimately periodic then $A^{\otimes(t+\sigma)}=\rho^{\sigma}A^{\otimes t}$
and hence  $A^{\otimes (t+\sigma)}\otimes x=\rho^{\sigma} A^{\otimes t}\otimes x$
holds for all $x\in\Rpn$ and all big enough $t$. Observe that
$\goodclass_3$ is a special case of $\goodclass_4$, which is a special case of $\goodclass_5$ since the
columns of matrix powers can be considered as orbits of the unit vectors.

To see that~\eqref{e:maingoal2} holds in all these cases, note that
in the column periodic case all column sequences $\{A_{\cdot
i}^t\}_{t\geq 1}$  end up with periodically repeating eigenvectors
of $A^{\otimes \sigma_i}$ or the zero vector, which implies that
$\spanmax(A^{\otimes t})\subseteq\bigoplus_{k\geq 1}
\vmax^{\Sigma}(A^{\otimes k})\subseteq\coremax(A)$ and hence
$\spanmax(A^{\otimes t})=\coremax(A)$ for all large enough $t$.
Thus, {\em finite stabilization of the core} occurs in all these classes. 
A necessary and sufficient condition for this finite stabilization is described in~\cite{BSS-inprep}.

\subsection{Core: a general argument}
\label{ss:pullman} The original argument of Pullman~\cite[Section
2]{Pul-71} used the separation of a point from a closed convex cone
by an open homogeneous halfspace (that contains the cone and does
not contain the point).

In the case of max algebra, Nitica and Singer~\cite{NS-07I} showed
that at each point $x\in\Rpn$ there are at most $n$ maximal
max-cones not containing this point. These {\em conical semispaces},
used to separate $x$ from any max cone not containing $x$, turn out
to be open. Hence they can be used in the max version of Pullman's
argument.

However, for the sake of a simultaneous proof we will exploit the
following analytic argument instead of separation. By
$B(x,\epsilon)$ we denote the intersection of the open ball centered
at $x\in\Rpn$ of radius $\epsilon$ with $\Rpn$. In the remaining
part of Section~\ref{s:core} we consider {\bf both algebras
simultaneously.}

\begin{lemma}
\label{l:analytic}
Let $x^1,\ldots,x^m\in\Rpn$ be nonzero and let $z\notin\spann(x^1,\ldots,x^m)$. Then there exists
$\epsilon>0$ such that $z\notin\spann(B(x^1,\epsilon),\ldots,B(x^m,\epsilon))$.
\end{lemma}
\begin{proof}
By contradiction assume that for each $\epsilon$ there exist points $y^i(\epsilon)\in B(x^i,\epsilon)$
and nonnegative scalars $\mu_i(\epsilon)$ such that
\begin{equation}
\label{e:epscomb}
z=\sum_{i=1}^m \mu_i(\epsilon) y^i(\epsilon).
\end{equation}
Since $y^i(\epsilon)\to x^i$ as $\epsilon\to 0$ and $x^i$ are nonzero,
we can assume that $y^i(\epsilon)$ are bounded from below by nonzero vectors $v^i$,
and then
$z\geq \sum_{i=1}^m \mu_i(\epsilon) v^i$ for all $\epsilon$,
implying that $\mu_i(\epsilon)$ are uniformly bounded from above.
By compactness we can assume that $\mu_i(\epsilon)$ converge to some $\mu_i\in\Rp$,
and then~\eqref{e:epscomb} implies by continuity that $z=\sum_{i=1}^m \mu_i x^i$, a contradiction.
\end{proof}

\begin{theorem}[{\cite[Theorem~2.1]{Pul-71}}]
\label{t:pullman} Assume that $\{K_l\}$ for $l\geq 1$, is a sequence
of cones in $\Rpn$ such that $K_{l+1}\subseteq K_l$ for all $l$, and
each of them generated by no more than $k$ nonzero vectors. Then the
intersection $K=\cap_{l=1}^{\infty}K_l$ is also generated by no more
than $k$ vectors.
\end{theorem}
\begin{proof}
Let $K_l=\spann(y^{l1},\ldots,y^{lk})$ (where some of the vectors
$y^{l1},\ldots,y^{lk}$ may be repeated when $K_l$ is generated by
less than $k$ nonzero vectors), and consider the sequences of
normalized vectors $\{y^{li}/||y^{li}||\}_{l\geq 1}$ for
$i=1,\ldots,k$, where $||u||:=\max u_i$ (or any other norm). As the
set $\{u\colon ||u||=1\}$ is compact, we can find a subsequence
$\{l_t\}_{t\geq 1}$ such that for $i=1,\ldots,k$, the sequence
$\{y^{l_ti}/||y^{l_ti}||\}_{t\geq 1}$ converges to a finite vector
$u^i$, which is nonzero since $||u^i||=1$. We will assume that
$||y^{l_t i}||=1$ for all $i$ and $t$.

We now show that $u^1,\ldots,u^k\in K$. Consider any $i=1,\ldots,k$.
For each $s$, $y^{l_t i}\in K_s$ for all sufficiently large $t$. As
$\{y^{l_ti}\}_{t\geq 1}$ converges to $u^i$ and $K_s$ is closed, we
have $u^i\in K_s$. Since this is true for each $s$, we have
$u^i\in\cap_{s=1}^{\infty} K_s=K$.

Thus $u^1,\ldots,u^k\in K$, and $\spann(u^1,\ldots,u^k)\subseteq K$.
We claim that also $K\subseteq \spann(u^1,\ldots, u^k)$. Assume to
the contrary that there is $z\in K$ that is not in
$\spann(u^1,\ldots, u^k)$. Then by Lemma~\ref{l:analytic} there
exists $\epsilon>0$ such that
$z\notin\spann(B(u^1,\epsilon),\ldots,B(u^k,\epsilon))$. Since the
sequence $\{y^{l_ti}\}_{t\geq 1}$ converges to $u^i$, we have
$y^{l_ti}\in B(u^i,\epsilon)$ for $t$ large enough, and
$$
\spann(y^{l_t1},\ldots,y^{l_tk})\subseteq \spann(B(u^1,\epsilon),\ldots,B(u^k,\epsilon))
$$
But $z$ belongs to $K_{l_t}=\spann(y^{l_t1},\ldots,y^{l_tk})$ since it belongs to the intersection of
all these cones, a contradiction.
\end{proof}

Theorem~\ref{t:pullman} applies to the sequence $\{\spann(A^t)\}_{t\geq 1}$,
so $\core(A)$ is generated by no more than $n$ vectors.

\begin{proposition}[{\cite[Lemma~2.3]{Pul-71}}]
\label{p:surj}
The mapping induced by $A$ on its core is a surjection.
\end{proposition}
\begin{proof}
First note that $A$ does induce a mapping on its core. If $z\in\core(A)$ then for each
$t$ there exists $x^t$ such that $z=A^tx^t$. Hence
$Az=A^{t+1}x^t$, so $Az\in\cap_{t\geq 2}\spann A^t=\core(A)$.

Next, let $m$ be such that $A^m$ has the greatest number of zero
columns (we assume that $A$ is not nilpotent; recall that a zero
column in $A^k$ remains zero in all subsequent powers). If
$z=A^tx^t$ for $t\geq m+1$, we also can represent it as
$A^{m+1}u^t$, where $u^t:=A^{t-m-1}x^t$. The components of $u^t$
corresponding to the nonzero columns of $A^{m+1}$ are bounded since
$A^{m+1}u^t=z$. So we can assume that the sequence of subvectors of
$u^{t}$ with these components converges. Then the sequence $y^t:=A^m
u^t$ also converges, since the indices of nonzero columns of $A^m$
coincide with those of $A^{m+1}$, which are the indices of the
converging subvectors of $u^t$. Let $y$ be the limit of $y^t$. Since
$y^s=A^{s-1} x^s$ are in $\spann(A^t)$ for all $s>t$, and since
$\spann(A^t)$ are closed, we obtain $y\in\spann(A^t)$ for all $t$.
Thus we found $y\in\core(A)$ satisfying $Ay=z$.
\end{proof}

Theorem~\ref{t:pullman} and Proposition~\ref{p:surj} show that the core is generated by finitely
many vectors in $\Rpn$ and that the mapping induced by $A$ on its core is ``onto''.

Now we use the fact that a finitely generated cone in the nonnegative orthant (and more generally, closed cone)
is generated by its extremals both in nonnegative algebra and in max algebra,
see~\cite{BSS,Wag-91}.

\begin{proposition}[{\cite[Theorem 2.2]{Pul-71}}]
\label{p:permute}
The mapping induced by $A$ on the extremal generators of its core is
a permutation (i.e., a bijection).
\end{proposition}
\begin{proof}
Let $\core(A)=\spann(u^1,\ldots,u^k)$ where $u^1,\ldots,u^k$ are extremals of the core.
Suppose that $x^j$ is a preimage of $u^j$ in the core, that is, $Ax^j=u^j$ for some $x^j\in\core(A)$,
$j=1,\ldots,k$.
Then $x^j=\sum_{i=1}^k \alpha_i u^i$ for some nonnegative coefficients $\alpha_1,\ldots,\alpha_k$,
and $u^j=\sum_{i=1}^k\alpha_i A u^i$. Since $u^j$ is extremal, it follows that $u^j$ is proportional
to $Au^i$ for some $i$. Thus for each $j\in\{1,\ldots,k\}$ there exists an $i\in\{1,\ldots,k\}$
such that $Au^i$ is a positive multiple of $u^j$. But since for each $i\in\{1,\ldots,k\}$ there is at most one $j$
such that $Au^i$ is a positive multiple of $u^j$, it follows that $A$ induces a bijection on the
set of extremal generators of its core.
\end{proof}

We are now ready to prove~\eqref{e:maingoal} and Main
Theorem~\ref{t:core} taking the periodicity of the eigencone
sequence (Main Theorem~\ref{t:periodicity}) in ``credit''.

\begin{proof}[Proof of Main Theorem~\ref{t:core}]
Proposition~\ref{p:permute} implies that all extremals of $\core(A)$
are eigenvectors of $A^q$, where $q$ denotes the order of the permutation
induced by $A$ on the extremals of $\core(A)$.
Hence $\core(A)$ is a subcone of the sum of all eigencones of all
powers of $A$, which is the inclusion relation~\eqref{e:maingoal2}. Combining this with the reverse inclusion of
Lemma~\ref{l:natural} we obtain that $\core(A)$ is precisely the sum
of all eigencones of all powers of $A$, and using
~\eqref{e:samespectrum} (proved in Section~\ref{s:sameaccess}
below), we obtain the first part of the equality of Main
Theorem~\ref{t:core}. The last part of the equality of Main
Theorem~\ref{t:core} now follows from the periodicity of eigencones
formulated in Main Theorem~\ref{t:periodicity}, or more precisely,
from the weaker result of Theorem~\ref{t:girls} proved in
Section~\ref{s:eigencones}.
\end{proof}

\section{Spectral classes and critical components of matrix powers}
\label{s:sameaccess}

This section is rather of technical importance. It shows that the
union of node sets of all spectral classes is invariant under matrix
powering, and that access relations between spectral classes in all
matrix powers are essentially the same. Further, the case of an
arbitrary eigenvalue can be reduced to the case of the principal
eigenvalue for all powers simultaneously (in both algebras). At the
end of the section we consider the critical components of
max-algebraic powers.

\subsection{Classes and access relations}

As in Section~\ref{s:core}, the arguments are presented in both
algebras simultaneously. This is due to the fact that the edge sets
of $\digr(A^{\otimes k})$ and $\digr(A^{\times k})$ are the same for
any $k$ and that the definitions of spectral classes in both
algebras are alike. Results of this section can be traced back, for
the case of nonnegative algebra, to the classical work of
Frobenius~\cite{Fro-1912}, see remarks on the very first page of
\cite{Fro-1912} concerning the powers of an irreducible nonnegative
matrix\footnote{Frobenius defines (what we could call) the cyclicity
or index of imprimitivity $k$ of an irreducible $S$ as the number of
eigenvalues that lie on the spectral circle. He then remarks ``If
$A$ is primitive, then every power of $A$ is again primitive and a
certain power and all subsequent powers are positive''. This is
followed by ``If $A$ is imprimitive, then $A^m$ consists of $d$
irreducible parts where $d$ is the greatest common divisor of $m$
and $k$. Further, $A^m$ is completely reducible.  The characteristic
functions of the components differ only in the powers of the
variable'' (which provides a converse to the preceding assertion).
And then ``The matrix $A^k$ is the lowest power of $A$ whose
components are all primitive''. The three quotations cover
Lemma~\ref{l:sameperron} in the case of nonnegative algebra.}.

The reader is also referred to the monographs of Minc~\cite{Minc},
Ber\-man-Plem\-mons~\cite{BP}, Brua\-ldi-Ry\-ser~\cite{BR}, and we
will often cite the work of Tam-Sch\-nei\-der~\cite[Section
4]{TS-94} containing all of our results in this section, in
nonnegative algebra.

\begin{lemma}[cf. {\cite[Ch.~5, Ex.~6.9]{BP}}, {\cite[Lemma~4.5]{TS-94}} ]
\label{l:sameperron}
Let $A$ be irreducible with the unique eigenvalue $\rho$,
let $\digr(A)$ have cyclicity $\sigma$ and $k$ be a positive integer.
\begin{itemize}
\item[{\rm (i)}] $A^k$ is permutationally similar to a direct sum of gcd$(k,\sigma)$ irreducible blocks
with eigenvalues $\rho^k$,
and $A^k$ does not have eigenvalues other than $\rho^k$.
\item[{\rm (ii)}] If $k$ is a multiple of $\sigma$, then the sets of indices in
these blocks coincide with the cyclic classes of $\digr(A)$.
\item[{\rm (iii)}] If $\supp(x)$ is a cyclic class of $\digr(A)$, then $\supp(Ax)$ is the previous
cyclic class.
 \end{itemize}
\end{lemma}
\begin{proof}
(i): Assuming without loss of generality $\rho=1$, let $X=\diag(x)$
for a positive eigenvector $x\in V(A,\rho)$ and consider
$B:=X^{-1}AX$ which is stochastic (nonnegative algebra), or
max-stochastic, i.e., such that $\bigoplus_{j=1}^n b_{ij}=1$ holds
for all $i$ (max algebra). By Theorem~\ref{t:brualdi}, $B^k$ is
permutationally similar to a direct sum of gcd$(k,\sigma)$
irreducible isolated blocks. These blocks are stochastic (or
max-stochastic), hence they all have an eigenvector $(1,\ldots,1)$
associated with the unique eigenvalue $1$. If $x\in
V(A^k,\Tilde{\rho})$ for some $\Tilde{\rho}$, then its subvectors
corresponding to the irreducible blocks of $A^k$ are also
eigenvectors of those blocks, or zero vectors. Hence
$\Tilde{\rho}=1$, which is the only eigenvalue of $A^k$.

(ii): By Theorem~\ref{t:brualdi}, $\digr(A)$ splits into
gcd$(k,\sigma)=\sigma$ components, and each of them contains exactly one cyclic class
of $\digr(A)$.

(iii): Use the definition of cyclic classes and that each node has an ingoing edge.
\end{proof}

\begin{lemma}
\label{l:samenonzero}
Both in max algebra and in nonnegative linear algebra,
the trivial classes of $A^k$ are the same for all $k$.
 \end{lemma}
\begin{proof}
In both algebras, an index belongs to a class with nonzero Perron root if and only if
the associated graph contains a cycle with a nonzero weight traversing the node with that
index. This property is invariant under taking matrix powers, hence the claim.
\end{proof}

In both algebras, each class $\mu$ of $A$ with cyclicity $\sigma$
corresponds to an irreducible submatrix $A_{\mu\mu}$. It is easy to
see that $(A^k)_{\mu\mu}=(A_{\mu\mu})^k$. Applying
Lemma~\ref{l:sameperron} to $A_{\mu\mu}$ we see that $\mu$ gives
rise to gcd$(k,\sigma)$ classes in $A^k$, which are said to be {\em
derived} from their common {\em ancestor} $\mu$. If $\mu$ is
trivial, then it gives rise to a unique trivial derived class of
$A^k$, and if $\mu$ is non-trivial then all the derived classes are
nontrivial as well. The classes of $A^k$ and $A^l$ derived from the
common ancestor will be called {\em related}. Note that this is an
equivalence relation on the set of classes of all powers of $A$.
Evidently, a class of $A^k$ is derived from a class of $A$ if and
only if its index set is contained in the index set of the latter
class. It is also clear that each class of $A^k$ has an ancestor in
$A$.

We now observe that access relations in matrix powers are
``essentially the same''. This has identical proof in max algebra
and nonnegative algebra.

\begin{lemma}
\label{l:sameaccess} Let $A\in\Rpnn$. For all $k,l\geq 1$ and
$\rho>0$, if an index $i\in\{1,\ldots,n\}$ accesses (resp.  is
accessed by) a class with Perron root $\rho^k$ in $A^k$ then $i$
accesses (resp. is accessed by) a related class with Perron root
$\rho^l$ in $A^l$.
\end{lemma}
\begin{proof} 

We deduce from Lemma~\ref{l:sameperron} and
Lemma~\ref{l:samenonzero} that the index set of each class of $A^k$
with Perron root $\rho^k$ is contained in the ancestor class of $A$
with Perron root $\rho$. Then, $i$ accessing (resp. being accessed
by) a class in $A^k$ implies $i$ accessing (resp. being accessed by)
its ancestor in $A$. Since $\rho>0$, this ancestor class is
nontrivial, so the access path can be extended to have a length
divisible by $l$, by means of a path contained in the ancestor
class. By Lemma~\ref{l:sameperron}, the ancestor decomposes in $A^l$
into several classes with the common Perron root $\rho^l$, and $i$
accesses (resp. is accessed by) one of them.
\end{proof}



\begin{theorem}[{\cite[Corollary 4.6]{TS-94}}]
\label{t:samespectrum}
Let $A\in\Rpnn$.
\begin{itemize}
\item[{\rm (i)}] If a class $\mu$ is spectral in $A$, then so are the classes
derived from it in $A^k$. Conversely, each spectral class of
$A^k$ is derived from a spectral class of $A$.
\item[{\rm (ii)}] For each class $\mu$ of $A$ with cyclicity $\sigma$, there are gcd$(k,\sigma)$ classes
of $A^k$ derived from it. If $k$ is a multiple of $\sigma$ then
the index sets of derived classes are the cyclic classes of
$\mu$.
\end{itemize}
\end{theorem}
\begin{proof}
(i): We will prove the following equivalent statement: For each pair
$\mu,\nu$ where $\mu$ is a class in $A$ and $\nu$ is a class derived
from $\mu$ in $A^k$, we have that $\mu$ is non-spectral if and only
if $\nu$ is non-spectral.

Observe that by Lemma~\ref{l:samenonzero}, the Perron root of $\mu$
is $0$ if and only if the Perron root of $\nu$ is $0$. In this case,
both $\mu$ and $\nu$ are non-spectral (by definition). Further, let
$\rho>0$ be the Perron root of $\mu$. Then, by
Lemma~\ref{l:sameperron}, the Perron root of $\nu$ is $\rho^k$. Let
$i$ be an index in $\nu$. It also belongs to $\mu$.

If $\mu$ is non-spectral, then $i$ is accessed in $A$ by a class
with Perron root $\rho'$ such that $\rho'>\rho$ in max algebra,
resp. $\rho'\geq\rho$ in nonnegative algebra. By
Lemma~\ref{l:sameaccess}, there is a class of $A^k$, which accesses
$i$ in $A^k$ and has Perron root $(\rho')^k$.  Since we have
$(\rho')^k>\rho^k$ in max algebra or resp. $(\rho')^k\geq \rho^k$ in
nonnegative algebra, we obtain that $\nu$, being the class to which
$i$ belongs in $A^k$, is also non-spectral.

Conversely, if $\nu$ is non-spectral, then $i$ is accessed in $A^k$
by a class $\theta$ with Perron root equal to $\Tilde{\rho}^k$ for
some $\Tilde{\rho}$, and such that $\Tilde{\rho}^k>\rho^k$ in max
algebra, resp. $\Tilde{\rho}^k\geq\rho^k$ in nonnegative algebra.
The ancestor of $\theta$ in $A$ accesses\footnote{This can be
observed immediately, or obtained by applying
Lemma~\ref{l:sameaccess}.} $i$ in $A$ and has Perron root
$\Tilde{\rho}$. Since we have $\Tilde{\rho}>\rho$ in max algebra or
resp. $\Tilde{\rho}\geq \rho$ in nonnegative algebra, we obtain that
$\mu$, being the class to which $i$ belongs in $A$, is also
non-spectral. Part (i) is proved.

(ii): This part follows directly from Lemma~\ref{l:sameperron} parts
(i) and (ii).
\end{proof}

\begin{corollary}
\label{c:samespectrum} Let $A\in\Rpnn$ and $k\geq 1$. Then
$\Lambda(A^k)=\{\rho^k\colon\rho\in\Lambda(A)\}.$
\end{corollary}
\begin{proof}
By Theorem~\ref{t:spectrum}, the nonzero eigenvalues of $A$ (resp.
$A^k$) are precisely the Perron roots of the spectral classes of $A$
(resp. $A^k$). By  Theorem~\ref{t:samespectrum}(i), if a class of
$A$ is spectral, then so is any class derived from it in $A^k$. This
implies
 that $\Lambda(A^k)\subseteq \{\rho^k\colon \rho\in\Lambda(A)\}.$ The
converse inclusion follows from 
the converse part of Theorem~\ref{t:samespectrum}(i).
\end{proof}

Let us note yet another corollary of Theorem~\ref{t:samespectrum}.
For $A\in\Rpnn$ and $\rho\geq 0$,
let $N(A,\rho)$ be the union of index sets of all classes of $A$
with Perron root $\rho$, and $N^s(A,\rho)$ be the union of index
sets of all {\bf spectral} classes of $A$ with Perron root $\rho$.
Obviously, $N^s(A,\rho)\subseteq N(A,\rho)$, and both sets (as
defined for arbitrary $\rho\geq 0$) are possibly empty.

\begin{corollary}
\label{c:specindex} Let $A\in\Rpnn$, $\rho\in\Rp$ and $k\geq 1$.
Then
\begin{itemize}
\item[{\rm (i)}] $N(A^k,\rho^k)=N(A,\rho)$,
\item[{\rm (ii)}] $N^s(A^k,\rho^k)=N^s(A,\rho)$.
\end{itemize}
\end{corollary}
\begin{proof}
(i): This part follows from Lemmas~\ref{l:sameperron}
and~\ref{l:samenonzero}.
(ii): Inclusion $N^s(A,\rho)\subseteq N^s(A^k,\rho^k)$ follows from
the direct part of Theorem~\ref{t:samespectrum}(i), and inclusion\\
$N^s(A^k,\rho^k)\subseteq N^s(A,\rho)$ follows from the converse
part of Theorem~\ref{t:samespectrum}(i).
\end{proof}


For the eigencones of $A\in\Rpnn$, the case of an arbitrary
$\rho\in\Lambda(A)$ can be reduced to the case of the principal
eigenvalue: $V(A,\rho)=V(A_{\rho},1)$ (Proposition~\ref{p:vamrho}).
Now we extend this reduction to the case of $V(A^k,\rho^k)$, for any
$k\geq 1$. As in the case of Propositon~\ref{p:vamrho}, we assume
that $A$ is in Frobenius normal form.

\begin{theorem}
\label{t:reduction}
Let $k\geq 1$ and $\rho\in \Lambda(A)$.
\begin{itemize}
\item[{\rm (i)}] The set of all indices having access to the spectral classes of $A^k$
with the eigenvalue $\rho^k$ equals $M_{\rho}$, for each $k$.
\item[{\rm (ii)}] $(A^k)_{\mrho\mrho}=\rho^k (\amrho)^k_{\mrho\mrho}$.
\item[{\rm (iii)}] $V(A^k,\rho^k)=V((\amrho)^k,1)$.
\end{itemize}
\end{theorem}
\begin{proof}
(i): Apply Corollary~\ref{c:specindex}~part~(ii) and
Lemma~\ref{l:sameaccess}. (ii): Use that $M_{\rho}$ is initial in
$\digr(A)$. (iii): By Proposition~\ref{p:vamrho} we have (assuming
that $A^k$ is in Frobenius normal form) that
$V(A^k,\rho^k)=V((A_{\rho^k})^k,1)$ where, instead of~\eqref{amrho},
\begin{equation}
\label{amrhok}
\begin{split}
& A^k_{\rho^k}:=\rho^{-k}
\begin{pmatrix}
0 & 0\\
0 & A^k_{M_{\rho}^kM_{\rho}^k}
\end{pmatrix},\ \text{and}\\
& M_{\rho}^k:=\{i\colon i\to\nu,\; \nu\ \text{is $(A^k,\rho^k)$-spectral}\}
\end{split}
\end{equation}
By part (i) $M_{\rho}^k=M_{\rho}$, hence $A_{\rho^k}^k=(\amrho)^k$ and the claim follows.
\end{proof}


\subsection{Critical components}

\label{ss:critcomp}

In max algebra, when $A$ is assumed to be strictly visualized, each
component $\maxmu$ of $\crit(A)$ with cyclicity $\sigma$ corresponds
to an irreducible submatrix $A^{[1]}_{\maxmu\maxmu}$ (as in the case
of classes, $A_{\maxmu\maxmu}$ is a shorthand for
$A_{N_{\maxmu}N_{\maxmu}}$). Using the strict visualization and
Lemma~\ref{l:CAk} we see that $(A^{\otimes
k})^{[1]}_{\maxmu\maxmu}=(A^{[1]}_{\maxmu\maxmu})^{\otimes k}$.
Applying Lemma~\ref{l:sameperron}(i) to $A^{[1]}_{\maxmu\maxmu}$ we
see that $\maxmu$ gives rise to gcd$(k,\sigma)$ critical components
in $A^{\otimes k}$. As in the case of classes, these components are
said to be derived from their common ancestor $\maxmu$.

Evidently a component of $\crit(A^{\otimes k})$ is derived from a
component of $\crit(A)$ if and only if its index set is contained in
the index set of the latter component. Following this line we now
formulate an analogue of Theorem~\ref{t:samespectrum} (and some
other results).

\begin{theorem}[cf. {\cite[Theorem 8.2.6]{But:10}}, {\cite[Theorem 2.3]{CGB-07}}]
\label{t:samecritical}
Let $A\in\Rpnn$.
\begin{itemize}
\item[{\rm (i)}] For each component $\maxmu$ of $\crit(A)$
with cyclicity $\sigma$, there are gcd$(k,\sigma)$ components of
$\crit(A^{\otimes k})$ derived from it. Conversely, each
component of $\crit(A^{\otimes k})$ is derived from a component
of $\crit(A)$. If $k$ is a multiple of $\sigma$, then index sets
in the derived components are the cyclic classes of $\maxmu$.
\item[{\rm (ii)}] The sets of critical indices
of $A^{\otimes k}$ for $k=1,2,\ldots$ are identical.
\item[{\rm (iii)}] If $A$ is strictly visualized, $x_i\leq 1$ for all $i$ and $\supp(x^{[1]})$ is a cyclic class
of $\maxmu$, then $\supp((A\otimes x)^{[1]})$ is the previous cyclic class of $\maxmu$.
\end{itemize}
\end{theorem}
\begin{proof}
(i),(ii): Both statements are based on the fact that
$\crit(A^{\otimes k})=(\crit(A))^k$, shown in Lemma~\ref{l:CAk}. To
obtain (i), also apply Theorem~\ref{t:brualdi} to a component
$\maxmu$ of $\crit(A)$. (iii): Use $(A\otimes
x)^{[1]}=A^{[1]}\otimes x^{[1]}$, the definition of cyclic classes
and the fact that each node in $\maxmu$ has an ingoing edge.
\end{proof}

\section{Describing extremals}
\label{s:extremals}

The aim of this section is to describe the extremals of the core, in
both algebras. To this end, we first give a parallel description of
extremals of eigencones (the Frobenius-Victory theorems).

\subsection{Extremals of the eigencones}
\label{ss:FV}

We now describe the principal eigencones in nonnegative linear
algebra and then in max algebra. By means of
Proposition~\ref{p:vamrho}, this description can be obviously
extended to the general case. As in Section~\ref{ss:pfelts}, both
descriptions are essentially known:
see~\cite{But:10,Fro-1912,Gau:92, Sch-86}.

We emphasize that the vectors $x^{(\mu)}$ and $x^{(\Tilde{\mu})}$ appearing below are
{\bf full-size}.

\begin{theorem}[Frobenius-Victory{\cite[Th.~3.7]{Sch-86}}]
\label{t:FVnonneg} Let $A\in\Rpnn$ have $\rhonon(A)=1$.
\begin{itemize}
\item[{\rm (i)}] Each spectral class $\mu$ with $\rhonon_{\mu}=1$ corresponds to
an eigenvector $x^{(\mu)}$, whose support consists of all
indices in the classes that have access to $\mu$, and all
vectors $x$ of $\vnon(A,1)$ with $\supp x=\supp x^{(\mu)}$ are
multiples of $x^{(\mu)}$.
\item[{\rm (ii)}] $\vnon(A,1)$ is generated by $x^{(\mu)}$ of {\rm (i)}, for $\mu$ ranging over all spectral
classes with $\rhonon_{\mu}=1$.
\item[{\rm (iii)}] $x^{(\mu)}$ of {\rm (i)} are extremals of $\vnon(A,1)$. {\rm (}Moreover,
$x^{(\mu)}$ are linearly independent.{\rm )}
\end{itemize}
\end{theorem}

Note that the extremality and the {\bf usual} linear independence of
$x^{(\mu)}$ (involving linear combinations with possibly negative
coefficients) can be deduced from the description of supports in
part (i), and from the fact that in nonnegative algebra, spectral
classes associated with the same $\rho$ do not access each other.
This linear independence also means that $\vnon(A,1)$ is a
simplicial cone. See also~\cite[Th.~4.1]{Sch-86}.

\begin{theorem}[{\cite[Th.~4.3.5]{But:10}}, {\cite[Th.~2.8]{SSB}}]
\label{t:FVmaxalg} Let $A\in\Rpnn$ have $\rhomax(A)=1$.
\begin{itemize}
\item[{\rm (i)}] Each component $\maxmu$ of $\crit(A)$ corresponds to
an eigenvector $x^{(\maxmu)}$ defined as one of the columns $A^*_{\cdot i}$
with $i\in N_{\maxmu}$, all columns with $i\in N_{\maxmu}$ being multiples of each other.
\item[{\rm (i')}] Each component $\maxmu$ of $\crit(A)$ is contained in a (spectral) class $\mu$
with $\rhomax_{\mu}=1$, and the support of each $x^{(\maxmu)}$
of {\rm (i)} consists of all indices in the classes that have
access to $\mu$.

\item[{\rm (ii)}] $\vmax(A,1)$ is generated by $x^{(\maxmu)}$ of {\rm (i)}, for $\maxmu$ ranging over all
components of $\crit(A)$.
\item[{\rm (iii)}] $x^{(\maxmu)}$ of {\rm (i)} are extremals in $\vmax(A,1)$. {\rm (}Moreover, $x^{(\maxmu)}$ are
strongly linearly independent in the sense of~\cite{But-03}.{\rm
)}
\end{itemize}
\end{theorem}

To verify (i'), not explicitly stated in the references, use (i) and the path interpretation of~$A^*$.

Vectors $x^{(\maxmu)}$ of Theorem~\ref{t:FVmaxalg} are also called the
{\em fundamental eigenvectors} of $A$, in max algebra.
Applying a strict visualization scaling (Theorem~\ref{t:strictvis}) allows us to get further details on
these fundamental eigenvectors.

\begin{proposition}[{\cite[Prop.~4.1]{SSB}}]
\label{p:xmuvis} Let $A\in\Rpnn$ be strictly visualized (in
particular, $\rhomax(A)=1$). Then
\begin{itemize}
\item[{\rm (i)}] For each component $\maxmu$ of $\crit(A)$,
$x^{(\maxmu)}$ of Theorem~\ref{t:FVmaxalg} can be canonically chosen as $A^*_{\cdot i}$ for any $i\in N_{\maxmu}$,
all columns with $i\in N_{\maxmu}$ being {\em equal} to each other.
\item[{\rm (ii)}] $x^{(\maxmu)}_i\leq 1$ for all $i$. Moreover,
$\supp(x^{(\maxmu)[1]})=N_{\maxmu}$.
\end{itemize}
\end{proposition}

\subsection{Extremals of the core}
\label{ss:extremals}

Let us start with the following observation in both algebras.

\begin{proposition}
\label{p:extremals}
For each $k\geq 1$, the set of extremals of $V^{\Sigma}(A^k)$ is the union of the sets of extremals
of $V(A^k,\rho^k)$ for $\rho\in\Lambda(A)$.
\end{proposition}
\begin{proof}
Due to the fact that
$\Lambda(A^k)=\{\rho^k\colon\rho\in\Lambda(A)\}$, we can assume
without loss of generality that $k=1$.

1. As $V^{\Sigma}(A)$ is the sum of $V(A,\rho)$ for $\rho\in\Lambda(A)$, it is generated by
the extremals of $V(A,\rho)$ for $\rho\in\Lambda(A)$. Hence each extremal of $V^{\Sigma}(A)$ is
an extremal of $V(A,\rho)$ for some $\rho\in\Lambda(A)$.

2. Let $x\in V(A,\rho_{\mu})$, for some spectral class $\mu$, be
extremal. Assume without loss of generality that $\rho_{\mu}=1$, and
by contradiction that there exist vectors $y^{\kappa}$, all of them
extremals of $V^{\Sigma}(A)$, such that $x=\sum_{\kappa}
y^{\kappa}$. By above, all vectors $y^{\kappa}$ are eigenvectors of
$A$. If there is $y^{\kappa}$ associated with an eigenvalue
$\rho_{\nu}>1$, then applying $A^t$ we obtain $x= (\rho_{\nu})^t
y^{\kappa}+\ldots$, which is impossible at large enough $t$. So
$\rho_{\nu}\leq 1$. With this in mind, if there is $y^{\kappa}$
associated with $\rho_{\nu}<1$, then 1) in nonnegative algebra we
obtain $Ax>A\sum_{\kappa} y^{\kappa}$, a contradiction; 2) in max
algebra, all nonzero entries of  $A\otimes y^{\kappa}$ go below the
corresponding entries of $x$ meaning that $y^{\kappa}$ is redundant.
Thus we are left only with $y^{\kappa}$ associated with
$\rhomax_{\nu}=1$, which is a contradiction: an extremal $x\in
\vmax(A,1)$ appears as a ``sum'' of other extremals of $\vmax(A,1)$
not proportional to $x$.
\end{proof}

A vector $x\in\Rpn$ is called {\em normalized} if $\max x_i=1$.
Recall the notation $\sigma_{\rho}$ introduced in Section~\ref{s:key}.

\begin{theorem}[cf.~{\cite[Theorem~4.7]{TS-94}}]
\label{t:tam-schneider}
Let $A\in\Rpnn$.
\begin{itemize}
\item[{\rm (i)}] The set of extremals of $\core(A)$ is the union of the sets of extremals of
$V(A^{\sigma_{\rho}},\rho^{\sigma_{\rho}})$ for all
$\rho\in\Lambda(A)$.
\item[{\rm (ii)}] {\bf In nonnegative algebra}, each spectral class $\mu$ with cyclicity $\sigma_{\mu}$
corresponds to a set of distinct $\sigma_{\mu}$ normalized
extremals of $\corenon(A)$, such that there exists an index in
their support that belongs to $\mu$, and each index in their
support has access to $\mu$.\\
{\bf In max algebra}, each critical component $\maxmu$ with
cyclicity $\sigma_{\Tilde{\mu}}$ associated with some
$\rho\in\Lmax(A)$ corresponds to a set of distinct
$\sigma_{\Tilde{\mu}}$ normalized extremals $x$ of
$\coremax(A)$, which are (normalized) columns of
$(\amrho^{\sigma_{\rho}})^*$ with indices in $N_{\maxmu}$.
\item[{\rm (iii)}] Each set of extremals described in {\rm (ii)} forms a simple
cycle under the action of $A$.
\item[{\rm (iv)}] There are no normalized extremals other than those described in {\rm (ii)}.
{\bf In nonnegative algebra,} the total number of normalized extremals equals the sum of cyclicities of all spectral classes of $A$.
{\bf In max algebra,} the total number of normalized extremals equals the sum of cyclicities of all
critical components of $A$.
\end{itemize}
\end{theorem}

\begin{proof}
(i) follows from Main Theorem~\ref{t:core} and
Proposition~\ref{p:extremals}.

For the proof of (ii) and (iii) we can fix
$\rho=\rho_{\mu}\in\Lambda(A)$, assume $A=\amrho$ (using
Theorem~\ref{t:reduction}) and $\sigma:=\sigma_{\rho}$. In max
algebra, we also assume that $A$ is strictly visualized.

(ii) {\bf In nonnegative algebra,} observe that by
Theorem~\ref{t:samespectrum}, each spectral class $\mu$ of $A$ gives
rise to $\sigma_{\mu}$ spectral classes in $A^{\times\sigma}$, whose
node sets are cyclic classes of $\mu$ (note that $\sigma_{\mu}$
divides $\sigma$). According to Frobenius-Victory
Theorem~\ref{t:FVnonneg}, these classes give rise to normalized
extremals of $\vnon(A^{\times\sigma},1)$, and the conditions on
support follow from Theorem~\ref{t:FVnonneg} and
Lemma~\ref{l:sameaccess}.

(iii): 
Let $x$ be an extremal described above. Then $\supp(x)\cap N_{\mu}$ is a cyclic class of $\mu$ and
$\supp(Ax)\cap N_{\mu}$ is the previous cyclic class of $\mu$, by Lemma~\ref{l:sameperron} part (iii).
It can be checked that all indices in $\supp(Ax)$ also have access to $\mu$.
By Proposition~\ref{p:permute}, $Ax$ is an extremal of $\corenon(A)$, and
hence an extremal of $\vnon(A^{\times\sigma},1)$. Theorem~\ref{t:FVnonneg} 
identifies $Ax$ with the extremal associated with the previous
cyclic class of $\mu$.

Vectors $x$, $Ax,\ldots,A^{\times\sigma_{\mu}-1}x$ are distinct since the intersections of their
supports with $N_{\mu}$ are disjoint, so they are exactly the set of extremals associated with $\mu$. Note that $A^{\times\sigma_{\mu}}x=x$, as $\supp(A^{\times\sigma_{\mu}}x)\cap N_{\mu}=\supp(x)\cap N_{\mu}$, and
both vectors are extremals of $\vnon(A^{\times\sigma},1)$.

(ii): {\bf In max algebra,} observe that by
Theorem~\ref{t:samecritical}(i) each component $\maxmu$ of
$\crit(A)$ gives rise to $\sigma_{\maxmu}$ components of
$\crit(A^{\otimes\sigma})$, whose node sets are the cyclic classes
of $\maxmu$ (note that $\sigma_{\maxmu}$ divides $\sigma$). These
components correspond to $\sigma_{\maxmu}$ columns of
$(A^{\otimes\sigma})^*$ with indices in different cyclic classes of
$\maxmu$, which are by Theorem~\ref{t:samecritical}(i) the node sets
of components of $\crit(A^{\otimes\sigma})$. By
Theorem~\ref{t:FVmaxalg} these columns of $(A^{\otimes\sigma})^*$
are extremals of $\vmax(A^{\otimes\sigma},1)$, and
Proposition~\ref{p:xmuvis}(ii) implies that they are normalized.

(iii): 
Let $x$ be an extremal described above. By
Proposition~\ref{p:xmuvis} and Theorem~\ref{t:samecritical}(i)
$\supp(x^{[1]})$ is a cyclic class of $\maxmu$, and by
Theorem~\ref{t:samecritical}(iii) $\supp((A\otimes x)^{[1]})$ is the
previous cyclic class of $\maxmu$. By Proposition~\ref{p:permute},
$A\otimes x$ is an extremal of $\coremax(A)$, and hence an extremal
of $\vmax(A^{\otimes\sigma},1)$. Proposition~\ref{p:xmuvis} 
identifies $A\otimes x$ with the extremal associated with the
previous cyclic class of $\maxmu$.

Vectors $x$, $A\otimes x,\ldots,A^{\otimes\sigma_{\maxmu}-1}x$ are distinct since
their booleanizations $x^{[1]}$, $(A\otimes x)^{[1]},\ldots,(A^{\otimes\sigma_{\maxmu}-1}\otimes x)^{[1]}$ are distinct,
 so they are exactly the set of extremals associated with $\maxmu$. Note that $A^{\otimes\sigma_{\maxmu}}\otimes x=x$, as
$(A^{\otimes\sigma_{\maxmu}}\otimes x)^{[1]}=x^{[1]}$
and both vectors are extremals of $\vmax(A^{\otimes\sigma},1)$.

(iv): {\bf In both algebras}, the converse part of
Theorem~\ref{t:samespectrum} (i) shows that there are no spectral
classes of $A^{\sigma}$ other than the ones derived from the
spectral classes of $A$. {\bf In nonnegative algebra,} this shows
that there are no extremals other than described in (ii). {\bf In
max algebra}, on top of that, the converse part of
Theorem~\ref{t:samecritical} (i) shows that there are no components
of $\crit(\amrho^{\otimes\sigma})$ other than the ones derived from
the components $\crit(\amrho)$, for $\rho\in\Lmax(A)$, hence there
are no extremals other than described in (ii). {\bf In both
algebras}, it remains to count the extremals described in (ii).
\end{proof}

\section{Sequence of eigencones}

\label{s:eigencones}

The main aim of this section is to investigate the periodicity of eigencones and to prove
Main Theorem~\ref{t:periodicity}. Unlike in Section~\ref{s:core}, the proof of periodicity will be
different for the cases of max algebra and nonnegative algebra. The
periods of eigencone sequences in max algebra and in nonnegative
linear algebra are also in general different, for the same
nonnegative matrix (see Section~\ref{s:examples} for an example). To
this end, recall the definitions of $\sigma_{\rho}$ and
$\sigma_{\Lambda}$ given in Section~\ref{s:key}, which will be used
below.

\subsection{Periodicity of the sequence}

\label{ss:period}


We first observe that in both algebras
\begin{equation}
\label{e:inclusion-eig}
\begin{split}
k\;\makebox{divides}\; l &\Rightarrow V(A^k,\rho^k)\subseteq V(A^l,\rho^l)\quad\forall\rho\in\Lambda(A),\\
k\;\makebox{divides}\; l &\Rightarrow V^{\Sigma}(A^k)\subseteq V^{\Sigma}(A^l).
\end{split}
\end{equation}

We now prove that the sequence of eigencones is periodic.

\begin{theorem}
\label{t:girls} Let $A\in\Rpnn$ and $\rho\in\Lambda(A)$.
\begin{itemize}
\item[{\rm (i)}]  $V(A^k,\rho^k)=V(A^{k+\sigma_{\rho}},\rho^{k+\sigma_{\rho}})$, and
$V(A^k,\rho^k)\subseteq V(A^{\sigma_{\rho}},\rho^{\sigma_{\rho}})$ for all $k\geq 1$.
\item[{\rm (ii)}] $V^{\Sigma}(A^k)=V^{\Sigma}(A^{k+\sigma_{\Lambda}})$ and $V^{\Sigma}(A^k)\subseteq V^{\Sigma}(A^{\sigma_{\Lambda}})$
for all $k\geq 1$.
\end{itemize}
\end{theorem}
\begin{proof}
We will give two separate proofs of part (i), for the case of max
algebra and the case of nonnegative algebra. Part (ii) follows from
part (i).

In both algebras, we can assume without loss of generality that
$\rho=1$, and using Theorem~\ref{t:reduction}, that this is the
greatest eigenvalue of $A$.

{\bf In max algebra}, by Theorem~\ref{t:nacht}, columns of
$A^{\otimes r}$ with indices in $N_c(A)$ are periodic for
$r\geq\Tc(A)$. Recall that by Corollary~\ref{TcaTcrit}, $\Tc(A)$ is
not less than $T(\crit(A))$, which is the greatest periodicity
threshold of the strongly connected components of $\crit(A)$. By
Theorem~\ref{t:BoolCycl} part (ii), $(\crit(A))^{t\sigma}$ consists
of complete graphs for $t\sigma\geq T(\crit(A))$, in particular, it
contains loops $(i,i)$ for all $i\in\critindices(A)$. Hence
$$
a_{ii}^{\otimes(t\sigma)}=1\quad\forall i\in\critindices(A),\ t\geq\lceil\frac{\Tc(A)}{\sigma}\rceil,
$$
and
$$
a_{ki}^{\otimes (l+t\sigma)}\geq a^{\otimes l}_{ki} a_{ii}^{\otimes (t\sigma)}=a_{ki}^{\otimes l}
\quad\forall i\in\critindices(A),\ \forall k,l,\  \forall t\geq\lceil\frac{\Tc(A)}{\sigma}\rceil,
$$
or, in terms of columns of matrix powers,
$$
A_{\cdot i}^{\otimes (l+t\sigma)}\geq A_{\cdot i}^{\otimes l}\quad
\forall i\in\critindices(A),\ \forall l,\ \forall t\geq\lceil\frac{\Tc(A)}{\sigma}\rceil.
$$
Multiplying this inequality repeatedly by $A^{\otimes l}$ we obtain
$A_{\cdot i}^{\otimes (kl+t\sigma)}\geq A_{\cdot i}^{\otimes(kl)}$
for all $k\geq 1$, and $A_{\cdot i}^{\otimes (k(l+t\sigma))}\geq
A_{\cdot i}^{\otimes(kl)}$ for all $k\geq 1$.
Hence we obtain
\begin{equation}
\label{e:*ineq}
(A^{\otimes (l+t\sigma)})^*_{\cdot i}\geq (A^{\otimes l})^*_{\cdot i}\quad\forall i\in\critindices(A),\ \forall l,\
\forall t\geq\lceil\frac{\Tc(A)}{\sigma}\rceil.
\end{equation}
On the other hand, using the ultimate periodicity of critical columns we have
\begin{equation*}
(A^{\otimes (l+t\sigma)})^*_{\cdot i}=\bigoplus\{A_{\cdot i}^{\otimes s}\colon  s\equiv kl\text{(mod} \sigma),\ k\geq 1,\ s\geq\Tc(A) \}
\end{equation*}
for all $l$ and all $t\sigma\geq\Tc(A)$, while generally
\begin{equation*}
\label{e:sweepineq}
(A^{\otimes l})^*_{\cdot i}\geq \bigoplus\{A_{\cdot i}^{\otimes s}\colon  s\equiv kl\text{(mod} \sigma),\ k\geq 1,\ s\geq\Tc(A) \}\ \forall l,
\end{equation*}
implying the reverse inequality with respect to~\eqref{e:*ineq}. It
follows that
\begin{equation}
\label{e:*eq}
(A^{\otimes (l+t\sigma)})^*_{\cdot i}= (A^{\otimes l})^*_{\cdot i}\quad\forall i\in\critindices(A),\ \forall l,\ \forall
t\geq\lceil\frac{\Tc(A)}{\sigma}\rceil,
\end{equation}
therefore $(A^{\otimes (l+\sigma)})^*_{\cdot i}=(A^{\otimes
(l+t\sigma+\sigma)})^*_{\cdot i}=(A^{\otimes (l+t\sigma)})^*_{\cdot
i} =(A^{\otimes l})^*_{\cdot i}$ for all critical indices $i$ and
all $l$. Since $V(A^{\otimes l},1)$ is generated by the critical
columns of $(A^{\otimes l})^*$, and the critical indices of
$A^{\otimes l}$ are $N_c(A)$ by Theorem~\ref{t:samecritical}(ii),
the periodicity $\vmax(A^{\otimes l},\rho^l)=\vmax(A^{\otimes
(l+\sigma)},\rho^{l+\sigma})$ follows. Using this
and~\eqref{e:inclusion-eig} we obtain $\vmax(A^{\otimes
l},\rho^l)\subseteq\vmax(A^{\otimes(l\sigma)},\rho^{l\sigma})=\vmax(A^{\otimes\sigma},\rho^{\sigma})$
for each $l$ and $\rho\in\Lmax(A)$.

{\bf In nonnegative algebra,} also assume that all final classes
(and hence only them) have Perron root $\rho=1$. Final classes of
$A^{\times l}$ are derived from the final classes of $A$; they (and
no other classes) have Perron root $\rho^l$. By
Theorem~\ref{t:samespectrum}(i) and~Corollary~\ref{c:div-boolean},
for any $t\geq 0$, the spectral classes of $A^{\times l}$ and
$A^{\times (l+t\sigma)}$ with Perron root $1$ have the same sets of
nodes, which we denote by $N_1,\ldots,N_m$ (assuming that their
number is $m\geq 1$).

By the Frobenius-Victory Theorem~\ref{t:FVnonneg}, the cone
$\vnon(A^{\times l},1)$ is generated by $m$ extremals
$x^{(1)},\ldots,x^{(m)}$ with the support condition of
Theorem~\ref{t:FVnonneg}(i) from which we infer that the subvectors
$x^{(\mu)}_{\mu}$ (i.e., $x^{(\mu)}_{N_{\mu}}$) are positive, while
$x^{(\mu)}_{\nu}$ (i.e., $x^{(\mu)}_{N_{\nu}}$) are zero for all
$\mu\neq \nu$ from $1$ to $m$, since the different spectral classes
by~\eqref{e:speclass} do not access each other, in the nonnegative
linear algebra. Analogously the cone
$\vnon(A^{\times(l+t\sigma)},1)$ is generated by $m$ eigenvectors
$y^{(1)},\ldots,y^{(m)}$ such that the subvectors $y^{(\mu)}_{\mu}$
are positive, while $y^{(\mu)}_{\nu}=0$ for all $\mu\neq \nu$ from
$1$ to $m$.

Assume first that $l=\sigma$. As $\vnon(A^{\times\sigma},1)\subseteq \vnon(A^{\times (t\sigma)},1)$, each $x^{(\mu)}$ is a nonnegative
 linear combination of $y^{(1)},\ldots,y^{(m)}$, and this implies $x^{(\mu)}=y^{(\mu)}$ for all $\mu=1,\ldots,m$.
Hence $\vnon(A^{\times (t\sigma)},1)=\vnon(A^{\times\sigma},1)$ for all $t\geq 0$.

We also obtain
$\vnon(A^{\times l},1)\subseteq \vnon(A^{\times (\sigma l)},1)=\vnon(A^{\times\sigma},1)$ for all $l$.
Thus $\vnon(A^{\times l},1)\subseteq\vnon(A^{\times (t\sigma)},1)$, and therefore
$\vnon(A^{\times l},1)\subseteq \vnon(A^{\times (l+t\sigma)},1)$. Now
if $\vnon(A^{\times l},1)$, resp. $\vnon(A^{\times (l+t\sigma)},1)$ are generated by $x^{(1)},\ldots,x^{(m)}$, resp.
$y^{(1)},\ldots, y^{(m)}$
described above and each  $x^{(\mu)}$ is a nonnegative
 linear combination of $y^{(1)},\ldots,y^{(m)}$,  this again implies $x^{(\mu)}=y^{(\mu)}$ for all $\mu=1,\ldots,m$,
and $\vnon(A^{\times (l+t\sigma)},1)=\vnon(A^{\times l},1)$ for all $t\geq 0$ and all $l$.

Using this and~\eqref{e:inclusion-eig} we obtain $\vnon(A^{\times l},\rho^l)\subseteq\vnon(A^{\times(l\sigma)},\rho^{l\sigma})=\vnon(A^{\times\sigma},\rho^{\sigma})$ for each $l$ and $\rho\in\Lnon(A)$.
\end{proof}


\if{
for the sum of all eigencones $V^{\Sigma}(A^k)=\sum_{\rho\in\Lambda(A)} V(A^k,\rho^k)$.
Summarizing
Theorems~\ref{t:girls-max} and~\ref{t:girls-nonneg} and using
$\sigma_{\rho}$ and $\sigma_{\Lambda}$
 we obtain the following formulation
(now uniting both algebras).

\begin{theorem}
\label{t:girls}
Let $A\in\Rpnn$. For all $l\geq 1$ and $\rho\in\Lambda(A)$,
\begin{itemize}
\item[(i)] $V(A^l,\rho^l)=V(A^{l+\sigma_{\rho}},\rho^{l+\sigma_{\rho}})$ and
$V(A^l,\rho^l)\subseteq V(A^{\sigma_{\rho}},\rho^{\sigma_{\rho}})$,
\item[(ii)] $V^{\Sigma}(A^l)=V^{\Sigma}(A^{l+\sigma_{\Lambda}})$ and
$V^{\Sigma}(A^l)\subseteq V^{\Sigma}(A^{\sigma_{\Lambda}})$.
\end{itemize}
\end{theorem}
}\fi

\subsection{Inclusion and divisibility}
\label{ss:incl}

We now show that the inclusion relations between the eigencones of
different powers of a matrix, in both algebras, strictly follow
divisibility of exponents of matrix powers with respect to
$\sigma_{\rho}$ and $\sigma_{\Lambda}$.
We start with a corollary of Theorem~\ref{t:girls}.

\begin{lemma}
\label{l:gcd-if}
Let $k,l\geq 1$ and $\rho\in\Lambda(A)$.
\begin{itemize}
\item[{\rm (i)}] $V(A^k,\rho^k)=V(A^{\text{gcd}(\sigma_{\rho},k)},\rho^{\text{gcd}(\sigma_{\rho},k)})$
and $V^{\Sigma}(A^k)=V^{\Sigma}(A^{\text{gcd}(\sigma_{\Lambda},k)})$.
\item[{\rm (ii)}] gcd$(k,\sigma_{\rho})=$ gcd$(l,\sigma_{\rho})$ implies $V(A^k,\rho^k)=V(A^l,\rho^l)$, and
gcd$(k,\sigma_{\Lambda})=$ gcd$(l,\sigma_{\Lambda})$ implies $V^{\Sigma}(A^k)=V^{\Sigma}(A^l)$.
\end{itemize}
\end{lemma}
\begin{proof}
(i): Let $\sigma:=\sigma_{\rho}$, and $s:=$gcd$(k,\sigma)$. If
$s=\sigma$ then $k$ is a multiple of $\sigma$ and
$V(A^k,\rho^k)=V(A^s,\rho^s)$ by Theorems~\ref{t:girls}(i).
Otherwise, since $s$ divides $k$, we have $V(A^s,\rho^s)\subseteq
V(A^k,\rho^k)$. In view of the periodicity
(Theorem~\ref{t:girls}(i)), it suffices to find $t$ such that
$V(A^k,\rho^k)\subseteq V(A^{s+t\sigma},\rho^{s+t\sigma})$. For
this, observe that $s+t\sigma$ is a multiple of $s=$
gcd$(k,\sigma)$. By Lemma~\ref{l:schur} (the Frobenius coin
problem), for big enough $t$ it can be expressed as $t_1 k +
t_2\sigma$ where $t_1,t_2\geq 0$. Moreover $t_1\neq 0$, for
otherwise we have $s=\sigma$. Then we obtain
\begin{equation*}
\begin{split}
V(A^k,\rho^k)&\subseteq V(A^{t_1k},\rho^{t_1k})=V(A^{t_1k+t_2\sigma},\rho^{t_1 k+t_2\sigma})\\
&=V(A^{s+t\sigma},\rho^{s+t\sigma})=V(A^s,\rho^s),
\end{split}
\end{equation*}
and the first part of the claim follows. The second part is obtained
similarly, using Theorem~\ref{t:girls}(ii) instead of
Theorems~\ref{t:girls}(i).

(ii) follows from (i).
\end{proof}

\begin{theorem}
\label{t:girls-major} Let $A\in\Rpnn$ and $\sigma$ be either the
cyclicity of a spectral class of $A$ {\bf (nonnegative algebra)} or
the cyclicity of a critical component of $A$ {\bf (max algebra)}.The
following are equivalent for all positive $k,l$:
\begin{itemize}
\item[{\rm (i)}] gcd$(k,\sigma)$ divides gcd$(l,\sigma)$ for all cyclicities $\sigma$;
\item[{\rm (ii)}] gcd$(k,\sigma_{\rho})$ divides gcd$(l,\sigma_{\rho})$ for all $\rho\in\Lambda(A)$;
\item[{\rm (iii)}] gcd$(k,\sigma_{\Lambda})$ divides gcd$(l,\sigma_{\Lambda})$;
\item[{\rm (iv)}] $V(A^k,\rho^k)\subseteq V(A^l,\rho^l)$ for all $\rho\in\Lambda(A)$ and
\item[{\rm (v)}] $V^{\Sigma}(A^k)\subseteq V^{\Sigma}(A^l)$.
\end{itemize}
\end{theorem}
\begin{proof}
(i)$\Rightarrow$(ii)$\Rightarrow$(iii) follow from elementary number theory.
(ii)$\Rightarrow$ (iv) and (iii)$\Rightarrow$(v)
follow from~\eqref{e:inclusion-eig} and Lemma~\ref{l:gcd-if} part (i) (which is essentially based on Theorem~\ref{t:girls}).
(iv)$\Rightarrow$(v) is trivial.
It only remains to show that (v)$\Rightarrow$ (i).

(v)$\Rightarrow$ (i): {\bf In both algebras}, take an extremal $x\in
V(A^k,\rho^k)$. As $V^{\Sigma}(A^k)\subseteq V^{\Sigma}(A^l)$, this
vector can be represented as $x=\sum_i y^i$, where $y^i$ are
extremals of $V^{\Sigma}(A^l)$. Each $y^i$ is an extremal of
$V(A^l,\Tilde{\rho}^l)$ for some $\Tilde{\rho}\in\Lambda(A)$ (as we
will see, only the extremals with $\Tilde{\rho}=\rho$ are
important). By Frobenius-Victory Theorems~\ref{t:FVnonneg}
and~\ref{t:FVmaxalg} and Theorem~\ref{t:samespectrum}(i), there is a
unique spectral class $\mu$ of $A$ to which all indices in
$\supp(x)$ have access. Since $\supp(y^i)\subseteq\supp(x)$, we are
restricted to the submatrix $A_{JJ}$ where $J$ is the set of all
indices accessing $\mu$ in $A$. In other words, we can assume
without loss of generality that $\mu$ is the only final class in
$A$, hence $\rho$ is the greatest eigenvalue, and $\rho=1$. Note
that $\supp(x)\cap N_{\mu}\neq\emptyset$.

{\bf In nonnegative algebra,}
restricting the equality $x=\sum_i y^i$ to $N_{\mu}$ we obtain
\begin{equation}
\label{e:eqsupps1}
\supp(x_{\mu})=\bigcup_i \supp(y^i_{\mu}).
\end{equation}
If $\supp(y^i_{\mu})$ is non-empty, then $y^i$ is associated with
a spectral class of $A^{\times l}$ whose nodes are in $N_{\mu}$.
Theorem~\ref{t:FVnonneg}(i) implies that $\supp(y^i_{\mu})$ consists
of all indices in a class of $A_{\mu\mu}^{\times l}$. As $x$ can be
any extremal eigenvector of $A^{\times k}$ with $\supp x\cap
N_{\mu}\neq\emptyset$, \eqref{e:eqsupps1} shows that each class of
$A_{\mu\mu}^{\times k}$ (corresponding to $x$) splits into several
classes of $A_{\mu\mu}^{\times l}$ (corresponding to $y^i$). By
Corollary~\ref{c:div-boolean} this is only possible when
gcd$(k,\sigma)$ divides gcd$(l,\sigma)$, where $\sigma$ is the
cyclicity of the spectral class $\mu$.

{\bf In max algebra,} since $\rho=1$, assume without loss of
generality that $A$ is strictly visualized. In this case $A$ and $x$
have all coordinates not exceeding $1$. Recall that $x^{[1]}$ is the
Boolean vector defined by $x^{[1]}_i=1$ $\Leftrightarrow$ $x_i=1$.
Vector $x$ corresponds to a unique critical component $\Tilde{\mu}$
of $\crit(A)$ with the node set $N_{\maxmu}$. Then instead
of~\eqref{e:eqsupps1} we obtain
\begin{equation}
\label{e:eqsupps2}
x^{[1]}=\bigoplus_i y^{i[1]}\quad \Rightarrow\quad \supp(x^{[1]}_{\maxmu})=\bigcup_i \supp(y^{i[1]}_{\maxmu}),
\end{equation}
where $\supp(x^{[1]})=\supp(x^{[1]}_{\maxmu})$ by
Proposition~\ref{p:xmuvis}(ii) and Theorem~\ref{t:samecritical}(i),
and hence also $\supp(y^{i[1]})=\supp(y^{i[1]}_{\maxmu})$. If
$\supp(y^{i[1]}_{\maxmu})$ is non-empty then also
$\supp(y^i_{N_{\mu}})$ is non-empty so that $y^i$ is associated with
the
eigenvalue $1$. 
\if{Lemma~\ref{l:sameperron} implies that
$\supp(y^{i[1]}_{N_{\maxmu}})$ comprises the set of indices in
several classes of $(A^{[1]}_{N_{\maxmu}N_{\maxmu}})^{\otimes l}$
(as it can be argued, in only one class if $y^{i}$ is extremal).
}\fi As $y^i$ is extremal, Proposition~\ref{p:xmuvis}(ii) implies
that $\supp(y^{i[1]}_{\maxmu})$ consists of all indices in a class
of $(A^{[1]}_{\maxmu\maxmu})^{\otimes l}$. As $x$ can be any
extremal eigenvector of $A^{\otimes k}$ with $\supp (x^{[1]})\cap
N_{\maxmu}\neq\emptyset$, \eqref{e:eqsupps2} shows that each class
of $(A^{[1]}_{\maxmu\maxmu})^{\otimes k}$  splits into several
classes of $(A^{[1]}_{\maxmu\maxmu})^{\otimes l}$. By
Corollary~\ref{c:div-boolean} this is only possible when
gcd$(k,\sigma)$ divides gcd$(l,\sigma)$, where $\sigma$ is the
cyclicity of the critical component $\Tilde{\mu}$.
\end{proof}

Let us also formulate the following version restricted to some $\rho\in\Lambda(A)$.

\begin{theorem}
\label{t:girls-minor} Let $A\in\Rpnn$, and let $\sigma$ be either
the cyclicity of a spectral class {\bf (nonnegative algebra)} or the
cyclicity of a critical component {\bf (max algebra)} associated
with some $\rho\in\Lambda(A)$. The following are equivalent for all
positive $k,l$:
\begin{itemize}
\item[{\rm (i)}] gcd$(k,\sigma)$ divides gcd$(l,\sigma)$ for all cyclicities $\sigma$;
\item[{\rm (ii)}] gcd$(k,\sigma_{\rho})$ divides gcd$(l,\sigma_{\rho})$;
\item[{\rm (iii)}] $V(A^k,\rho^k)\subseteq V(A^l,\rho^l)$.
\end{itemize}
\end{theorem}
\begin{proof}
(i)$\Rightarrow$(ii) follows from the elementary number theory, and
(ii)$\Rightarrow$(iii) follows from~\eqref{e:inclusion-eig} and
Lemma~\ref{l:gcd-if}(i). The proof of (iii)$\Rightarrow$(i) follows
the lines of the proof of Theorem~\ref{t:girls-major}
(v)$\Rightarrow$ (i), with a slight simplification that
$\Tilde{\rho}=\rho$ and further, $x$ and all $y^i$ in $x=\sum_i y^i$
are associated with the same eigenvalue.
\end{proof}

We are now ready to deduce Main Theorem~\ref{t:periodicity}

\begin{proof}[Proof of Main Theorem~\ref{t:periodicity}]

We prove the first part. The inclusion $V(A^k,\rho^k)\subseteq
V(A^{\sigma},\rho^{\sigma})$ was proved in Theorem~\ref{t:girls}
(i), and we are left to show that $\sigma_{\rho}$ is the least such
$p$ that $V(A^{k+p},\rho^{k+p})=V(A^k,\rho^k)$ for all $k\geq 1$.
But taking $k=\sigma_{\rho}$ and using Theorem~\ref{t:girls-minor}
(ii)$\Leftrightarrow$(iii), we obtain
gcd$(\sigma_{\rho}+p,\sigma_{\rho})=\sigma_{\rho}$, implying that
$\sigma_{\rho}$ divides $\sigma_{\rho}+p$, so $\sigma_{\rho}$
divides $p$. Since Theorem~\ref{t:girls} (i) also shows that
$V(A^{k+\sigma_{\rho}},\rho^{k+\sigma_{\rho}})=V(A^k,\rho^k)$ for
all $k\geq 1$, the result follows.

The second part can be proved similarly, using
Theorem~\ref{t:girls}(ii) and Theorem~\ref{t:girls-major}
(iii)$\Leftrightarrow$(v).
\end{proof}

\section{Examples}
\label{s:examples}
We consider two examples of reducible nonnegative matrices, examining their
core in max algebra and in nonnegative linear algebra.

\noindent {\em Example 1.} Take
\begin{equation}
A=
\begin{pmatrix}
 0.1206   &      0   &      0   &      0   &      0\\
 0.5895  &  0.2904 &   1  &  0.8797  &  0.4253\\
 0.2262   & 0.6171  &  0.3439  &  1   & 0.3127\\
 0.3846  &  0.2653  &  0.5841  &  0.2607   & 1\\
 0.5830  &  1   & 0.1078  &  0.5944 &   0.1788
\end{pmatrix}\,\enspace .
\end{equation}
$A$ has two classes with node sets $\{1\}$ and $\{2,3,4,5\}$.
Both in max algebra and in nonnegative linear algebra, the only spectral class arises from $M:=\{2,3,4,5\}$. The max-algebraic
Perron root of this class is $\rhomax(A)=1$, and the critical graph
consists of just one cycle $2\to 3\to 4\to 5\to 2$.

The eigencones $\vmax(A,1)$, $\vmax(A^{\otimes 2},1)$, $\vmax(A^{\otimes 3},1)$ and
$\vmax(A^{\otimes 4},1)$ are generated by the last four columns of
the Kleene stars $A^*$, $(A^{\otimes 2})^*$,
$(A^{\otimes 3})^*$, $(A^{\otimes 4})^*$. Namely,
\begin{equation*}
\begin{split}
\vmax(A,1)&=\vmax(A^{\otimes 3},1)=\spanmax\{(0\ 1\ 1\ 1\ 1)\},\\
\vmax(A^{\otimes 2},1)=\spanmax &\{(0,1,0.8797,1,0.8797),\ (0,0.8797,1,0.8797,1)\},\\
\vmax(A^{\otimes 4},1)=\spanmax &\{(0,1,0.6807,0.7738,0.8797),\ (0,0.8797,1,0.6807,0.7738),\\
 &(0,0.7738,0.8797,1,0.6807),\ (0,0.6807,0.7738,0.8797,1)\}
\end{split}
\end{equation*}

By Main Theorem~\ref{t:core}, $\coremax(A)$ is equal to
$\vmax(A^{\otimes 4},1)$. Computing the max-algebraic powers of $A$
we see that the sequence of submatrices $A^{\otimes t}_{MM}$ becomes
periodic after $t=10$, with period $4$. In particular,
\begin{equation}
A^{\otimes 10}=
\begin{pmatrix}
 \alpha    &    0    &     0   &      0  &       0\\
    0.4511  &  0.7738 &   0.6807 &   1  &  0.8797\\
    0.5128  &  0.8797  &  0.7738  &  0.6807  &  1\\
    0.5830  &  1  &  0.8797  &  0.7738  &  0.6807\\
    0.5895  &  0.6807  &  1  &  0.8797  &  0.7738
\end{pmatrix},
\end{equation}
where $0<\alpha<0.0001$. Observe that the last four columns are
precisely the ones that generate $\vmax(A^{\otimes 4},1)$. Moreover, if
$\alpha$ was $0$ then the first column would be the following max-combination of the last four columns:
$$
a^{\otimes 10}_{41}A^{\otimes 10}_{\cdot 2}\oplus a^{\otimes 10}_{51}A^{\otimes 10}_{\cdot 3}\oplus
a^{\otimes 10}_{21}A^{\otimes 10}_{\cdot 4}\oplus a^{\otimes 10}_{31}A^{\otimes 10}_{\cdot 5}.
$$
On the one hand, the first column of $A^{\otimes t}$
cannot be a max-combination of the last four columns for any $t>0$ since
$a_{11}^{\otimes t}>0$. On the other hand,
$a_{11}^{\otimes t}\to 0$ as $t\to\infty$ ensuring
that the first column belongs to the core ``in the limit''.

Figure~\ref{f:petersfigure1} gives a symbolic illustration of what is going on in this example.

\begin{figure}
\centering
\vskip -1cm
\includegraphics[width=0.8\linewidth]{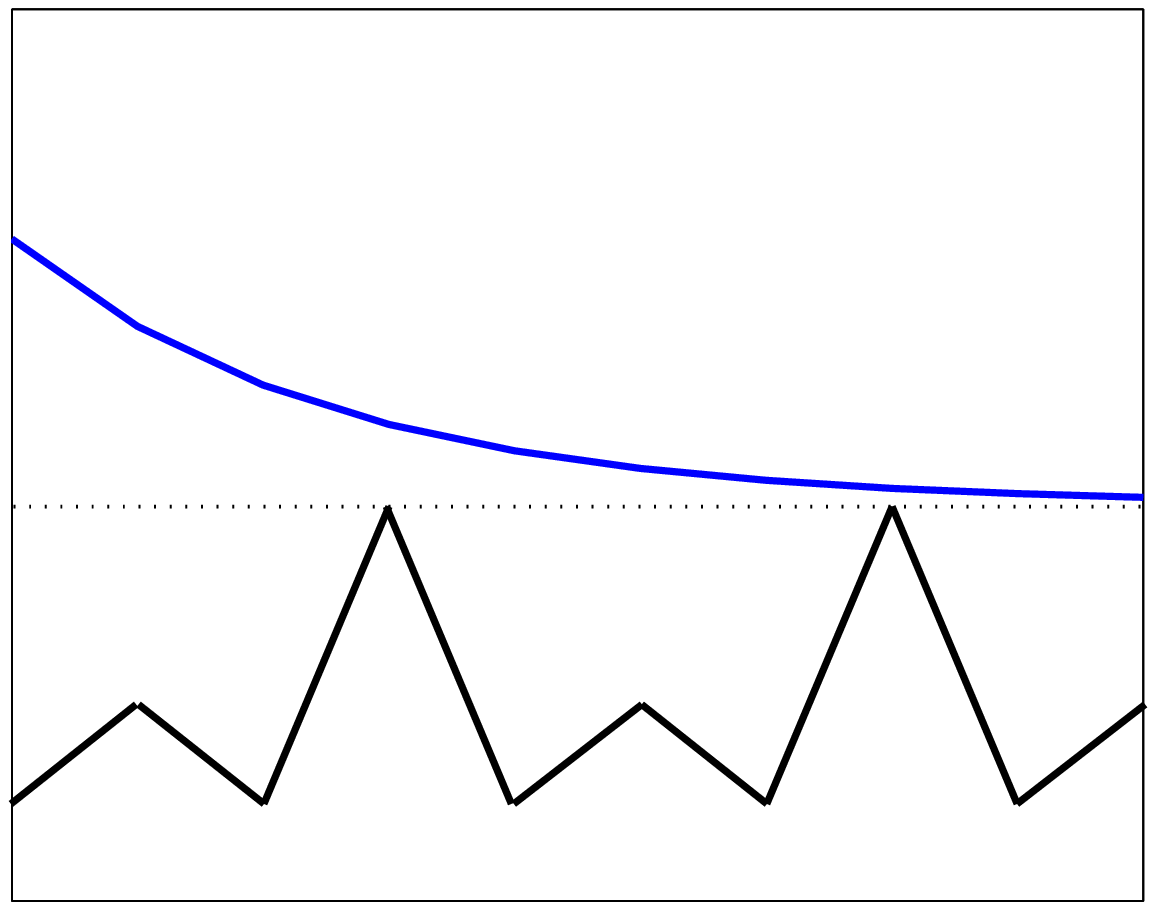}
\vskip -1cm
\caption{The spans of matrix powers (upper curve) and the periodic sequence of their eigencones (lower graph)
in Example 1 (max algebra).}
\label{f:petersfigure1}
\end{figure}

\if{
\begin{figure}
\centering
\includegraphics[width=0.8\linewidth]{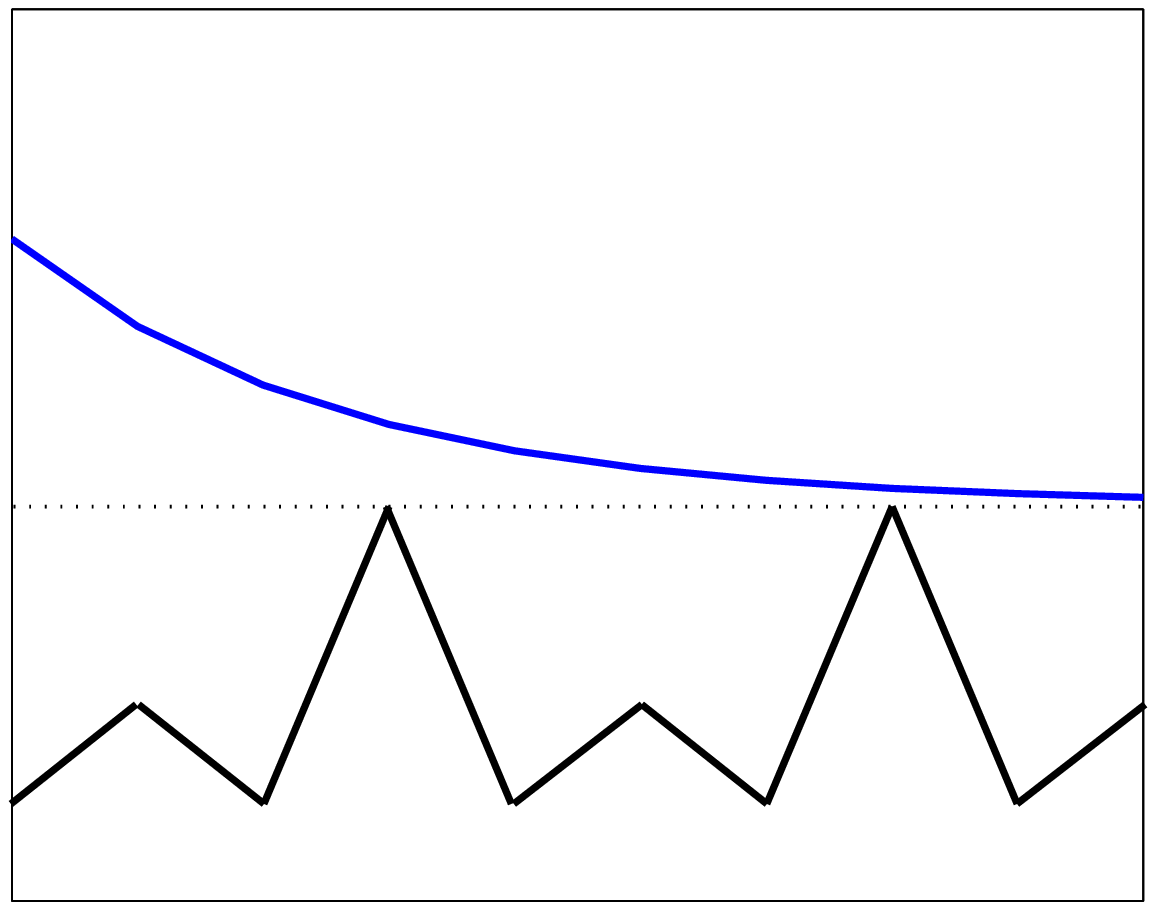}
\caption{The spans of matrix powers (upper curve) and the periodic sequence of their eigencones
(lowergraph)
in Example 1 (max algebra).}
\label{f:petersfigure1}
\end{figure}
}\fi

In {\bf nonnegative algebra}, the block $A_{MM}$ with $M=\{2,3,4,5\}$ is also the only spectral block .
Its Perron root is approximately $\rhonon(A)=2.2101$, and the corresponding eigencone is
$$
\vnon(A,\rhonon(A))=\spannon\{(0,\ 0.5750,\ 0.5107,\ 0.4593,\ 0.4445)\}.
$$
Taking the usual powers of $(A/\rhonon(A))$ we see that
$$
\left(A/\rhonon(A)\right)^{\times 12}=
\begin{pmatrix}
 \alpha      &   0    &     0      &   0     &    0\\
    0.2457   & 0.2752  &  0.2711  &  0.3453   & 0.2693\\
    0.2182  &  0.2444  &  0.2408   & 0.3067  &  0.2392\\
    0.1963   & 0.2198  &  0.2165  &  0.2759  &  0.2151\\
    0.1899  &  0.2127 &   0.2096  &  0.2670  &  0.2082
\end{pmatrix},
$$
where $0<\alpha<0.0001$, and that the first four digits of all entries in all higher powers are the same.
It can be verified that the submatrix $(A/\rhonon(A))^{\times 12}_{MM}$ is, approximately, the outer product of
the Perron eigenvector with itself, while the first column is also almost proportional to it.

\noindent {\em Example 2.}  Take
\begin{equation}
A=
\begin{pmatrix}
      0  &  1  &   0   &   0\\
    1  & 0   &    0   &   0\\
    0.6718 &  0.2240 &   0.5805  &  0.1868\\
    0.6951 &   0.6678  &  0.4753  &  0.3735
\end{pmatrix}\,\enspace .
\end{equation}

This matrix has two classes $\mu$ and $\nu$ with index sets $\{1,2\}$ and $\{3,4\}$,
and both classes are spectral, in both algebras. In max algebra
$\rhomax_{\mu}=1$ and $\rhomax_{\nu}=a_{33}<1$.
The eigencones of matrix powers associated with
$\rhomax_{\mu}=1$ are
\begin{equation*}
\begin{split}
\vmax(A,1)&=\spanmax\{(1,1,0.6718,0.6951)\},\\
\vmax(A^{\otimes 2},1)&=
\spanmax\{(1,0,0.3900,0.6678),\ (0,1,0.6718,0.6951)\},\\
\end{split}
\end{equation*}
and the eigencone associated with $\rhomax_{\nu}$ is generated by the
third column of the matrix:
$$
\vmax(A,\rhomax_{\nu})=\spanmax\{(0,\ 0,\ 0.5805,\ 0.4753)\}.
$$

By Main Theorem~\ref{t:core}, $\coremax(A)$ is equal to the
(max-algebraic) sum of $\vmax(A^{\otimes 2},1)$ and
$\vmax(A,\rhomax_{\nu})$. To this end, observe that already in the
second max-algebraic power
\begin{equation}
A^{\otimes 2}=
\begin{pmatrix}
1     &    0     &    0     &    0\\
         0  &  1     &    0    &     0\\
    0.3900  &  0.6718  &  0.3370  &  0.1084\\
    0.6678  &  0.6951  &  0.2759  &  0.1395
\end{pmatrix}
\end{equation}
the first two columns are the generators of $\vmax(A^{\otimes 2},1)$.
However, the last column is still not proportional to the third one
which shows that $\spanmax(A^{\otimes 2})\neq \coremax(A)$. However, it can be checked
that this happens in $\spanmax(A^{\otimes 4})$, with the first two columns still equal to the
generators of  $\vmax(A^{\otimes 2},1)$, which shows that $\spanmax(A^{\otimes 4})$
is the sum of above mentioned max cones, and hence
$\spanmax(A^{\otimes 4})=\spanmax(A^{\otimes 5})=\ldots=\coremax(A)$. Hence we see that $A$ is column
periodic ($\goodclass_5$) and the core finitely
stabilizes. See Figure~\ref{f:petersfigure2} for a symbolic illustration.

\begin{figure}
\centering
\vskip -1cm
\includegraphics[width=0.8\linewidth]{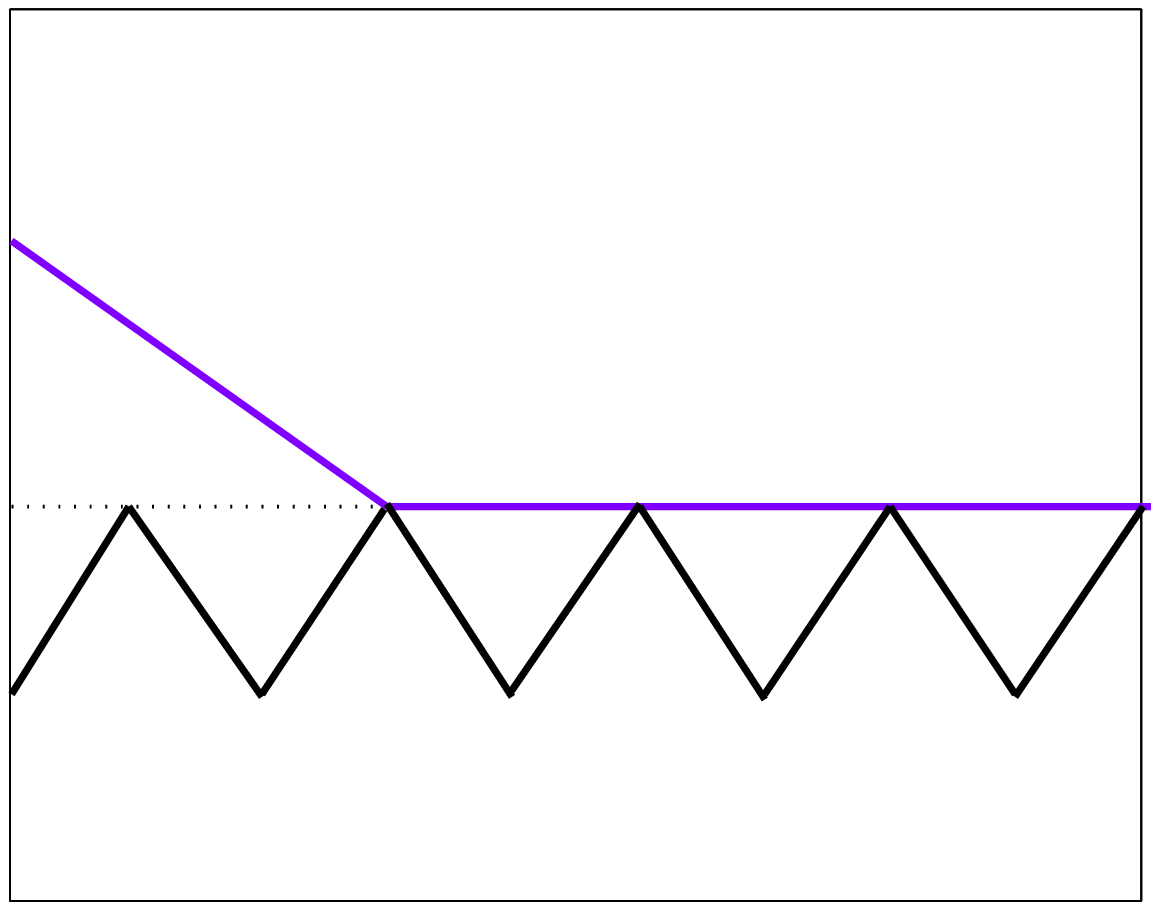}
\vskip -1cm
\caption{The spans of matrix powers (upper graph) and the periodic sequence of their eigencones (lower graph)
in Example 2 (max algebra)}
\label{f:petersfigure2}
\end{figure}

\if{
\begin{figure}
\centering
\includegraphics[width=0.7\linewidth]{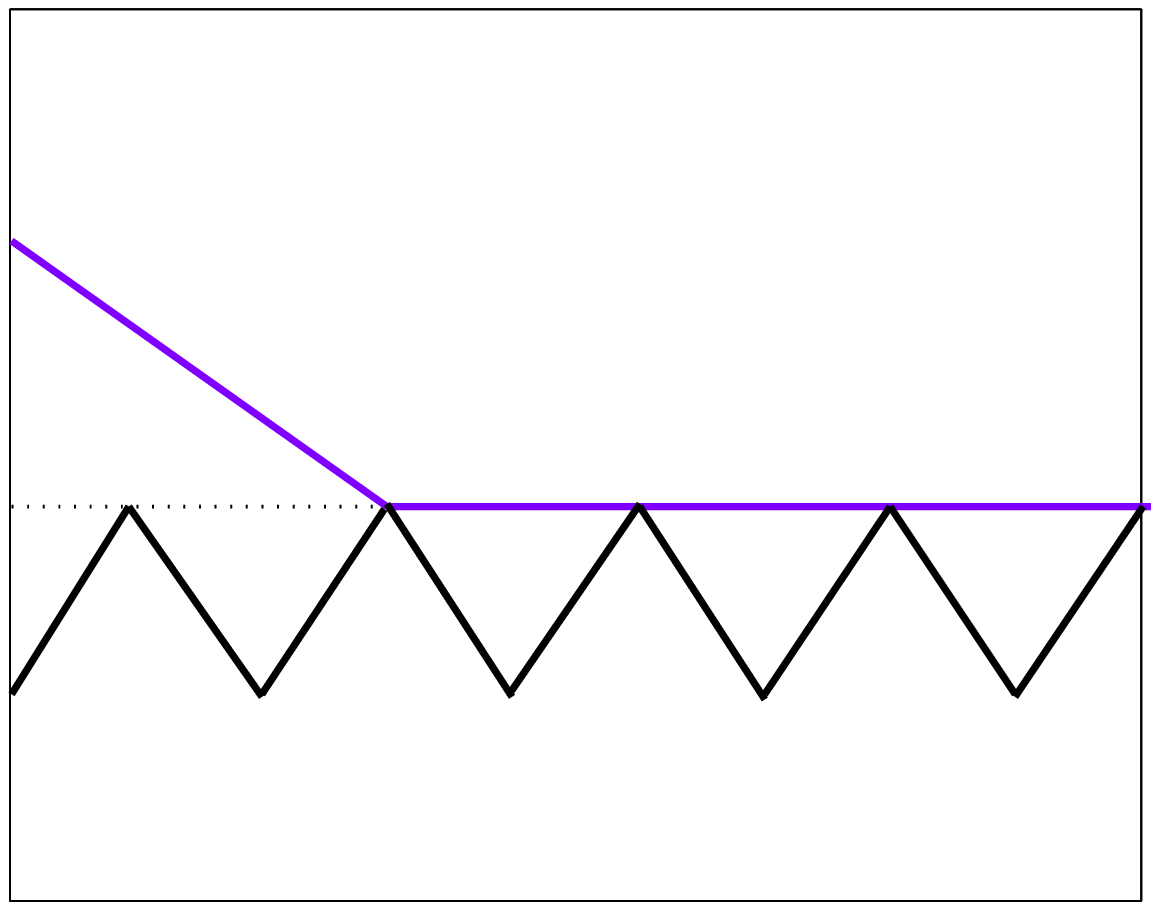}
\caption{The spans of matrix powers (upper curve) and the periodic sequence of their
eigencones (lower graph)
in Example 2 (max algebra)}
\label{f:petersfigure2}
\end{figure}
}\fi

In {\bf nonnegative algebra}, $\rhonon_{\mu}=1$ and
$\rhonon_{\nu}=0.7924$. Computing the eigenvectors of $A$ and
$A^{\times 2}$ yields
\begin{equation*}
\begin{split}
\vnon(A,1)&=\spannon\{(0.1326,0.1326,0.6218,0.7604)\},\\
\vnon(A^{\times 2},1)&=
\spannon\{(0.2646,0,0.5815,0.7693),\ (0,0.2566,0.6391,0.7251)\},\\
\end{split}
\end{equation*}
and
$$
\vnon(A,\rhonon_{\nu})=\spannon\{(0,\ 0,\ 0.6612,\ 0.7502)\}.
$$

Here $\corenon(A)$ is equal to the ordinary (Minkowski) sum of
$\vnon(A^{\times 2},1)$ and $\vnon(A,\rhonon_{\nu})$. To this end,
it can be observed that, within the first $4$ digits, the first two
columns of $A^{\times t}$ become approximately periodic after
$t=50$, and the columns of powers of the normalized submatrix
$A_{\nu\nu}/\rhonon_{\nu}$ approximately stabilize after $t=7$. Of
course, there is no finite stabilization of the core in this case.
However, the structure of the nonnegative core is similar to the
max-algebraic counterpart described above.


\section*{Acknowledgement}

The authors are grateful to the anonymous referee for careful
reading and a great number of useful suggestions which helped to
improve the presentation and to eliminate some flaws and errors.

\end{document}